\documentclass[a4paper,english]{extarticle}
\usepackage[latin9]{inputenc}
\usepackage{babel}
\usepackage{mathrsfs}
\usepackage{amsthm}
\usepackage{amsmath}
\usepackage{amssymb}
\usepackage{esint}
\usepackage[unicode=true,pdfusetitle,
 bookmarks=true,bookmarksnumbered=false,bookmarksopen=false,
 breaklinks=false,pdfborder={0 0 1},backref=false,colorlinks=false]
 {hyperref}

\makeatletter

\pdfpageheight\paperheight
\pdfpagewidth\paperwidth

\numberwithin{equation}{section}
\numberwithin{figure}{section}
\theoremstyle{plain}
\newtheorem{thm}{\protect\theoremname}[section]
\theoremstyle{remark}
\newtheorem{rem}[thm]{\protect\remarkname}
\theoremstyle{plain}
\newtheorem{lem}[thm]{\protect\lemmaname}
\theoremstyle{definition}
\newtheorem{defn}[thm]{\protect\definitionname}
\theoremstyle{plain}
\newtheorem{prop}[thm]{\protect\propositionname}
\ifx\proof\undefined
\newenvironment{proof}[1][\protect\proofname]{\par
\normalfont\topsep6\p@\@plus6\p@\relax
\trivlist
\itemindent\parindent
\item[\hskip\labelsep\scshape #1]\ignorespaces
}{%
\endtrivlist\@endpefalse
}
\providecommand{\proofname}{Proof}
\fi
\theoremstyle{plain}
\newtheorem{cor}[thm]{\protect\corollaryname}

\usepackage{cancel}
\numberwithin{equation}{section}
\date{}

\makeatother

\providecommand{\corollaryname}{Corollary}
\providecommand{\definitionname}{Definition}
\providecommand{\lemmaname}{Lemma}
\providecommand{\propositionname}{Proposition}
\providecommand{\remarkname}{Remark}
\providecommand{\theoremname}{Theorem}

\begin{document}

\title{The Calderón problem for variable coefficients nonlocal elliptic
operators}

\author{Tuhin Ghosh\thanks{{\footnotesize{}{}{}Institute for Advanced Study, Jockey Club, HKUST,
Hong Kong, China }}, Yi-Hsuan Lin\thanks{{\footnotesize{}{}{}Department of Mathematics, University of Washington,
Seattle, USA}}, Jingni Xiao\thanks{{\footnotesize{}{}{}Department of Mathematics, Hong Kong Baptist
University, Hong Kong, China}}}
\maketitle
\begin{abstract}
In this paper, we introduce an inverse problem of a Schrödinger type
variable nonlocal elliptic operator $(-\nabla\cdot(A(x)\nabla))^{s}+q)$,
for $0<s<1$. We determine the unknown bounded potential $q$ from
the exterior partial measurements associated with the nonlocal Dirichlet-to-Neumann
map for any dimension $n\geq2$. Our results generalize the recent
initiative \cite{ghosh2016calder} of introducing and solving inverse
problem for fractional Schrödinger operator $((-\Delta)^{s}+q)$ for
$0<s<1$. We also prove some regularity results of the direct problem
corresponding to the variable coefficients fractional differential
operator and the associated degenerate elliptic operator. 
\end{abstract}
\textbf{Key words.} The Calderón problem, nonlocal Schrödinger equation,
anisotropic, unique continuation principle, Runge approximation property,
degenerate elliptic equations, $A_{p}$ weight, Almgren's frequency
function, doubling inequality\\
\textbf{Mathematics Subject Classification}: 35R30, 26A33, 35J10,
35J70

\section{Introduction}

Let $\mathcal{L}$ be an elliptic partial differential operator. We
consider in this paper an inverse problem associated to the nonlocal
fractional operator $\mathcal{L}^{s}$ with the power $s\in(0,1)$.
We introduce the corresponding Calderón problem of determining the
unknown bounded potentials $q(x)$ from the exterior measurements
on the Dirichlet-to-Neumann (DN) map of the nonlocal Schrödinger equation
$(\mathcal{L}^{s}+q)u=0$. It intends to generalize the recent study
on the Calderón problem for the fractional Schrödinger equation \cite{ghosh2016calder}.
The study of the nonlocal operators is currently an active research
area in mathematics and often covers vivid problems coming from different
fields including mathematical physics, finance, biology, geology.
See the references \cite{bucur2016nonlocal,ros2015nonlocal} for subsequent
discussions. The study of inverse problems remains as a popular field
in applied mathematics since A.P. Calderón published his pioneering
work ``On an inverse boundary value problem'' \cite{calderon2006inverse}
in 1980s. The problem proposed by Calderón is: ``Is it possible to
determine the electrical conductivity of a medium by making voltage
and current measurements on its boundary?'' It gets its momentum
with the seminal work of Sylvester and Uhlmann \cite{sylvester1987global},
solving the Calderón problem in space dimension. Following that, many
Calderón's type inverse problems and related questions have been addressed
and extensively studied by many authors, mostly related to the local
operators (for example: $\mathcal{L}$). In a very recent progress
the study of Calderón's type inverse problem is being initiated for
nonlocal operators, in particular the Calderón problem of the fractional
Schrödinger operator $(-\Delta)^{s}+q(x)$ has been addressed in \cite{ghosh2016calder}.

In this article, we continue the progress by considering more general
nonlocal operators $(-\nabla\cdot(A(x)\nabla))^{s}+q(x)$ where $A$
is possibly variable coefficient \textit{anisotropic} matrix with
standard ellipticity and boundedness assumptions on it. This work
also offers a comparative study between nonlocal inverse problem of
$(-\nabla\cdot(A(x)\nabla))^{s}+q(x)$ for $0<s<1$, and the local
inverse problem of $-\nabla\cdot(A(x)\nabla)+q(x)$. The solvability
of the local inverse problem is fully known in two dimension, whereas
in three and higher dimension it is partially solved for certain class
of anisotropic matrix. We will see such difficulties do not arise
in our non-local analogue.

In this paper, we consider $\mathcal{L}$ to be a second order linear
elliptic operator of the divergence form 
\begin{equation}
\mathcal{L}:=-\nabla\cdot(A(x)\nabla),\label{eq: Local elliptic operator}
\end{equation}
which is defined in the entire space $\mathbb{R}^{n}$ for $n\geq2$,
where $A(x)=(a_{ij}(x))$, $x\in\mathbb{R}^{n}$ is an $n\times n$
symmetric matrix satisfying the ellipticity condition, i.e., 
\begin{equation}
\begin{cases}
a_{ij}=a_{ji}\mbox{ for all }1\leq i,j\leq n,\mbox{ and }\\
\Lambda^{-1}|\xi|^{2}\leq\sum_{i,j=1}^{n}a_{ij}(x)\xi_{i}\xi_{j}\leq\Lambda|\xi|^{2}\mbox{ for all }x\in\mathbb{R}^{n},\mbox{ for some }\Lambda>0.
\end{cases}\label{eq:ellipticity and symmetry condition}
\end{equation}
Our definition of the fractional power $\mathcal{L}^{s}$, with its
domain $\mathrm{Dom}(\mathcal{L}^{s})$, initiated from the spectral
theorem. We then extend the operator $\mathcal{L}^{s}$, by applying
the heat kernel and its estimates, as a bounded linear operator 
\[
\mathcal{L}^{s}:H^{s}(\mathbb{R}^{n})\longrightarrow H^{-s}(\mathbb{R}^{n}).
\]
The detailed definition of $\mathcal{L}^{s}$ is included in Section~\ref{Section 2}.

If $\Omega$ is a bounded open set in $\mathbb{R}^{n}$, let us consider
$u\in H^{s}(\mathbb{R}^{n})$ a solution to the Dirichlet problem
\begin{equation}
(\mathcal{L}^{s}+q)u=0\quad\mbox{ in }\Omega,\quad u|_{\Omega_{e}}=g,\label{eq:nonlocalproblem}
\end{equation}
where $q=q(x)\in L^{\infty}(\Omega)$ and $\Omega_{e}$ is the exterior
domain denoted by 
\[
\Omega_{e}=\mathbb{R}^{n}\backslash\overline{\Omega},
\]
and it is assumed that int$(\Omega_{e})\neq\emptyset$. We also assume
that $0$ is not an eigenvalue of the operator $(\mathcal{L}^{s}+q)$,
which means 
\begin{equation}
\begin{cases}
\mbox{ if }w\in H^{s}(\mathbb{R}^{n})\mbox{ solves }(\mathcal{L}^{s}+q)w=0\mbox{ in }\Omega\mbox{ and }w|_{\Omega_{e}}=0,\\
\mbox{ then }w\equiv0.
\end{cases}\label{eq:eigenvalue condition}
\end{equation}
For being $q\geq0$, the condition \eqref{eq:eigenvalue condition}
is satisfied. Then for any given $g\in H^{s}(\Omega_{e})$, there
exists a unique solution $u\in H^{s}(\mathbb{R}^{n})$ solves the
nonlocal problem \eqref{eq:nonlocalproblem} (see Proposition \ref{props:wellposedness}).
Next, we are going to define the associated DN map of the problem
\eqref{eq:nonlocalproblem} in an analogues way introduced in \cite[Lemma 2.4]{ghosh2016calder}
as 
\begin{equation}
\Lambda_{q}:X\to X^{*},\label{eq:Dirichlet-to-Neumann}
\end{equation}
where $X$ is the abstract trace space $X=H^{s}(\mathbb{R}^{n})/\widetilde{H}^{s}(\Omega)$
defined by 
\begin{equation}
(\Lambda_{q}[g],[h])=\mathcal{B}_{q}(u,h),\mbox{ for }g,h\in H^{s}(\mathbb{R}^{n}),\label{eq:Associated DN map}
\end{equation}
the $[\cdot]$ stands for the equivalence class in $X$, i.e., for
given $g\in H^{s}(\mathbb{R}^{n})$, 
\[
[g]=g+\widetilde{g},\mbox{ with }\widetilde{g}\in\widetilde{H}^{s}(\Omega),
\]
and $\mathcal{B}_{q}(\cdot,\cdot)$ in \eqref{eq:Associated DN map}
is the standard bilinear form associated to the above problem \eqref{eq:nonlocalproblem}
explicitly introduced in Section 2.2.

The range of the DN map could be interpreted as infinitesimal amount
of particles migrating to the exterior domain $\Omega_{e}$ in the
steady state free diffusion process in $\Omega$ modeled by \eqref{eq:nonlocalproblem}
which gets excited due to some source term in $\Omega_{e}$. Analogue
to the diffusion process, similar interpretations might be regarded
in the theory of stochastic analysis. For more details, see \cite{andreu2010nonlocal,chen2006traces,piiroinen2015probabilistic}.

Furthermore, if the domain $\Omega$, the potential $q$ in $\Omega$,
the source term in $\Omega_{e}$ and the matrix $A(x)$ in \eqref{eq: Local elliptic operator}
satisfying \eqref{eq:ellipticity and symmetry condition} in $\mathbb{R}^{n}$
are sufficiently smooth, the DN map is more explicit and is given
by (see Remark \ref{smooth DN map Lemma}) 
\[
\Lambda_{q}:H^{s+\beta}(\Omega_{e})\to H^{-s+\beta}(\Omega_{e})\mbox{ with }\Lambda_{q}g=\mathcal{L}^{s}u|_{\Omega_{e}},
\]
for any $\beta\geq0$ satisfying $\beta\in(s-\frac{1}{2},\frac{1}{2})$.
Heuristically, given an open set $W\subseteq\Omega_{e}$, we interpret
$\Lambda_{q}g|_{W}$ as measuring the cost required to maintain the
exterior value $g$ in $W$ for the fixed inhomogeneity in the system
given by $A(x)$ in the whole space $\mathbb{R}^{n}$.

The following theorem is the main result in this article. It is a
generalization of the fractional Schrödinger inverse problem studied
in \cite{ghosh2016calder} in any dimension $n\geq2$. This is also
a local data result with exterior Dirichlet and Neumann measurements
in arbitrary open (possibly disjoint) sets $\mathcal{O}_{1},\mathcal{O}_{2}\subseteq\Omega_{e}$.

\subsubsection*{($\mathcal{H}$) Hypothesis on $A(x)$}
\begin{enumerate}
\item $A(x)$ is a bounded matrix-valued function in $\mathbb{R}^{n}$ satisfying
\eqref{eq:ellipticity and symmetry condition}. 
\item Let $A(x)\in C^{\infty}(\mathbb{R}^{n})$.\end{enumerate}
\begin{thm}
\label{thm: Main}For $n\geq2$, let $\Omega\subseteq\mathbb{R}^{n}$
be a bounded domain with Lipschitz boundary and let $q_{1},q_{2}\in L^{\infty}(\Omega)$
satisfy condition \eqref{eq:eigenvalue condition}. Assume that $\mathcal{O}_{1},\mathcal{O}_{2}\subseteq\Omega_{e}$
are arbitrary open sets and $\Lambda_{q_{j}}$ is the DN map with
respect to $(\mathcal{L}^{s}+q_{j})u=0$ in $\Omega$ for $j=1,2$.
If 
\begin{equation}
\left.\Lambda_{q_{1}}g\right|_{\mathcal{O}_{2}}=\left.\Lambda_{q_{2}}g\right|_{\mathcal{O}_{2}}\mbox{ for any }g\in C_{c}^{\infty}(\mathcal{O}_{1}),\label{eq:Equal exteiror measurements}
\end{equation}
and $A(x)$ satisfies the hypothesis ($\mathcal{H}$), then one can
conclude that 
\[
q_{1}=q_{2}\mbox{ in }\Omega.
\]

\end{thm}
Theorem \ref{thm: Main} can be interpreted as a partial data result
for the above nonlocal inverse problem. Analogues resembles can be
made with the study of the partial data Calderón's type problem, the
richness of such works can be found in \cite{imanuvilov2010calderon,isakov1990inverse,kenig2014calderon,kenig2007calderon}.

Let us present a comparative study between our non-local inverse problem
and the known local inverse problem. We begin with recalling the following
local inverse problem as: Determining the uniqueness of the potentials
$q_{1}=q_{2}$ in $\Omega$ from the information on the associated
DN maps $\Lambda_{A,q_{1}}=\Lambda_{A,q_{2}}$ on $\partial\Omega$,
where the $\Lambda_{A,q_{j}}:H^{1/2}(\partial\Omega)\to H^{-1/2}(\partial\Omega)$
is the DN map defined by $\Lambda_{q}(u|_{\partial\Omega})=(A\nabla u)\cdot\nu|_{\partial\Omega}$
(where $\nu$ is the unit outer normal on $\partial\Omega$), and
$u_{j}$ solves 
\[
(\mathcal{L}+q_{j})u_{j}=-\nabla\cdot(A(x)\nabla u_{j})+q_{j}(x)u_{j}=0\mbox{ in }\Omega\mbox{ for }j=1,2,
\]
with $A\in L^{\infty}(\Omega)$ satisfying the ellipticity condition
(\ref{eq:ellipticity and symmetry condition}).

It has been answered positively in two dimensional case by using the
isothermal coordinate. For $n\geq3$, the answer is known for a certain
class of anisotropic matrices $A$. This problem has been often addressed
via geometry settings which goes as follows: Let $(M,g)$ be a oriented
compact Riemannian $n$-dimensional manifold with $C^{\infty}$-smooth
boundary $\partial M$ and let $q$ be a continuous potential on $M$.
Consider 
\begin{equation}
(-\Delta_{g}+q)u=0\mbox{ in }M,\label{eq:Laplace-Beltrami Calderon}
\end{equation}
where 
\[
\Delta_{g}=\sum_{j,k=1}^{n}g^{-1/2}\dfrac{\partial}{\partial x^{j}}\left(g^{1/2}g^{jk}\dfrac{\partial}{\partial x^{k}}\right)
\]
is the Laplace-Beltrami operator on $(M,g)$ and $g=\det(g_{jk})$
with $(g_{jk})=(g^{jk})^{-1}$. If $\{0\}$ is not an eigenvalue of
$-\Delta_{g}+q$, we have the corresponding DN map on $\partial M$
defined by 
\[
\Lambda_{g,q}:H^{1/2}(\partial M)\to H^{-1/2}(\partial M)\mbox{ by }\Lambda_{g,q}(u|_{\partial M}):=\left.\sum_{j,k=1}^{n}g^{jk}\dfrac{\partial u}{\partial x_{j}}\nu_{k}\right|_{\partial M},
\]
where $\nu=(\nu_{1},\nu_{2},\cdots,\nu_{n})$ is the unit outer normal
on $\partial M$. The connection between the matrix $A=(a_{jk})$
and the metric $g^{jk}$ can be made as 
\[
g^{jk}(x)=(\det A(x))^{-1/(n-2)}a_{jk}(x)\mbox{ for }n\geq3.
\]
In the two-dimensional setting, if $\Lambda_{g,q_{1}}=\Lambda_{g,q_{2}}$
on $\partial M$, then $q_{1}=q_{2}$ in $M$ whenever $q_{1}$, $q_{2}$
are continuous potential on $M$, see \cite{guillarmou2011calderon}.
However in the case of three and higher dimensions, it has been answered
partially. Under special geometries, for instance, when $(M,g)$ is
admissible (see \cite[Definition 1.5]{ferreira2009limiting} and $q_{1}$,
$q_{2}$ are $C^{\infty}$-smooth, then $\Lambda_{g,q_{1}}=\Lambda_{g,q_{2}}$
on $\partial M$ implies $q_{1}=q_{2}$ in $M$, see \cite[Theorem 1.6]{ferreira2009limiting}.

In our paper, we study the inverse problem associated with the nonlocal
operator $\mathcal{L}^{s}+q$, where $\mathcal{L}=-\nabla\cdot(A(x)\nabla)$
and $s\in(0,1)$. We can determine $q_{1}=q_{2}$ in $\Omega\subseteq\mathbb{R}^{n}$
for any $n\geq2$ via the partial information $\left.\Lambda_{q_{1}}g\right|_{\mathcal{O}_{2}}=\left.\Lambda_{q_{2}}g\right|_{\mathcal{O}_{2}}$
for any $g\in C_{c}^{\infty}(\mathcal{O}_{1})$, with $\mathcal{O}_{1}$,
$\mathcal{O}_{2}$ being arbitrary open subsets in $\Omega_{e}$,
for any matrix-valued function $A(x)$ satisfying the hypothesis ($\mathcal{H}$).
Note that we do not assume any further special structures on $A(x)$
unlike to the case $s=1$, for example, the method (see \cite{ferreira2009limiting})
consists of considering the limiting Carleman weight function for
the Laplace-Beltrami operator in $M$ and constructing the corresponding
complex geometrical optics (CGO) solutions based on those weights
of the problem \eqref{eq:Laplace-Beltrami Calderon}. Whereas, our
analysis relies on the Runge type approximation result (cf. Theorem
\ref{thm:(Approximation-theorem)}) based on the strong uniqueness
property (cf. Theorem \ref{thm:(Runge-approximation-property)}) of
the nonlocal operator $\mathcal{L}^{s}$.

For $A(x)$ being an $n\times n$ identity matrix $I_{n}$, then $\mathcal{L}$
becomes the Laplacian operator $(-\Delta)$ and the associated inverse
problem for $s=1$ has been studied extensively. When $n\geq3$, the
global uniqueness result is due to \cite{sylvester1987global} for
$q\in L^{\infty}$ and the authors \cite{chanillo1990problem,nachman1992inverse}
proved it for the case of $q\in L^{p}$. When $n=2$, Bukhgeim \cite{bukhgeim2008recovering}
proved it for slightly more regular potentials and see \cite{blaasten2015stability}
for the case of $q\in L^{p}$. We refer readers to \cite{uhlmann2014seenunseen}
for detailed survey on this inverse problem. For $s\in(0,1)$, the
study of this problem has been recently initiated in \cite{ghosh2016calder}.

Let us briefly mention the way we prove the uniqueness result $q_{1}=q_{2}$
in $\Omega$ as stated in Theorem \ref{thm: Main}. By having the
following integral identity 
\[
\int_{\Omega}(q_{1}-q_{2})u_{1}u_{2}\,dx=0,
\]
which we obtain from the assumption on the DN maps \eqref{eq:Equal exteiror measurements},
in particular, by taking $u_{j}\in H^{s}(\mathbb{R}^{n})$ solving
$(\mathcal{L}^{s}+q_{j})u_{j}=0$ in $\Omega$ with $\mathrm{supp}(u_{j})\subset\overline{\Omega}\cup\overline{\mathcal{O}{}_{j}}$;
finally we derive for any $g\in L^{2}(\Omega)$ 
\[
\int_{\Omega}(q_{1}-q_{2})g\,dx=0.
\]
The proof of the above integral identity will be completed with subsequent
requirements of the following strong uniqueness property and the Runge
approximation property for the nonlocal operator $\mathcal{L}^{s}$,
similar to the results known (see \cite{ghosh2016calder}) for the
fractional Laplacian operator. 
\begin{thm}
\label{thm:(Approximation-theorem)}(Strong uniqueness property) Let
$u\in H^{s}(\mathbb{R}^{n})$ be the function with $u=\mathcal{L}^{s}u=0$
in some open set $\mathcal{O}$ of $\mathbb{R}^{n}$, where $s\in(0,1)$
and $A(x)$ satisfies the hypothesis ($\mathcal{H}$), then $u\equiv0$
in $\mathbb{R}^{n}$. 
\end{thm}

\begin{thm}
\label{thm:(Runge-approximation-property)}(Runge approximation property)
Let $\Omega\subseteq\mathbb{R}^{n}$ be a bounded open set and $D\subseteq\mathbb{R}^{n}$
be an arbitrary open set containing $\Omega$ such that $\mathrm{int}(D\backslash\overline{\Omega})\neq\emptyset$.
If $A(x)$ satisfies the hypothesis ($\mathcal{H}$) and $q\in L^{\infty}(\Omega)$
satisfies \eqref{eq:eigenvalue condition}, then for any $f\in L^{2}(\Omega)$,
for any $\epsilon>0$, we can find a function $u_{\epsilon}\in H^{s}(\mathbb{R}^{n})$
which solves 
\[
(\mathcal{L}^{s}+q)u_{\epsilon}=0\mbox{ in }\Omega\mbox{ and supp}(u_{\epsilon})\subseteq\overline{D}
\]
and 
\[
\|u_{\epsilon}-f\|_{L^{2}(\Omega)}<\epsilon.
\]

\end{thm}
The paper is organized as follows. In Section~\ref{Section 2}, we
will give a brief review of the background knowledge required in our
paper, including the definition of the operator $\mathcal{L}^{s}$.
Some results for the Dirichlet problem, including the well-posedness
and the definition of the corresponding DN map, associated with the
nonlocal operator $\mathcal{L}^{s}$ will be established in Section~\ref{Section 3}.
In Section \ref{Section 4}, we will show that the nonlocal problem
in $\mathbb{R}^{n}$ is related to a extension degenerate local elliptic
problem in $\mathbb{R}^{n}\times(0,\infty)$, which was first characterized
by \cite{stinga2010extension}. We also introduce suitable regularity
results for the nonlocal operator $\mathcal{L}^{s}$ in $\mathbb{R}^{n}$,
and its extension operator in $\mathbb{R}^{n}\times(0,\infty)$. These
regularity results play the essential role to achieve our desired
results. We hope that this could be of some independent interests.
In Section \ref{Section 5}, we will derive the strong unique continuation
property (SUCP) for variable fractional operators and we prove Theorems
\ref{thm:(Approximation-theorem)} and \ref{thm:(Runge-approximation-property)}.
In Section \ref{Section 6}, we prove the  nonlocal type  Calderón
problem, Theorem \ref{thm: Main}. In Appendix, we offer the proof
of the existence, uniqueness, and related properties including the
Almgren type frequency function method and the associated doubling
inequality for the degenerate elliptic problem.

\section*{Acknowledgment}

The authors would like to thank Professor Gunther Uhlmann for suggesting
the problem, and also to thank Professor Mikko Salo, Professor Gunther
Uhlmann and Doctor Hui Yu for helpful discussions. Y. H. L. is partially
supported by MOST of Taiwan under the project 160-2917-I-564-048.

\section{Preliminaries\label{Section 2}}

In this section, we will discuss some key properties for the variable
coefficients fractional nonlocal operator $\mathcal{L}^{s}=(-\nabla\cdot(A(x)\nabla))^{s}$.
For $A(x)$ being an identity matrix, the operator $\mathcal{L}^{s}$
becomes the well-known fractional Laplacian operator $(-\Delta)^{s}$,
and the detailed study about the $(-\Delta)^{s}$ is available in
\cite{bogdan1997boundary,cabre2014nonlinear,caffarelli2007extension,caffarelli2016fractional,ros2012pohozaev,ros2014dirichlet,ruland2015unique,seeleycomplex,silvestre2007regularity}.

\subsection{Spectral Theory}

\label{sec:SpectralTheory} We sketch in this section some basis of
the spectral theory which will be used in this paper. For details,
readers can refer to the references \cite{riesz1990functional,rudin1991functional,taylor1968functions},
etc. 

Let $\mathcal{L}$ be a non-negative definite and self-adjoint operator
densely defined in a Hilbert space, say, $L^{2}(\mathbb{R}^{n})$.
Let $\phi$ be a real-valued measurable function defined on the spectrum
of $\mathcal{L}$. Then the following defined $\phi(\mathcal{L})$
is also a self-adjoint operator in $L^{2}(\mathbb{R}^{n})$, 
\[
\phi(\mathcal{L}):=\int_{0}^{\infty}\phi(\lambda)\,dE_{\lambda},
\]
where $\{E_{\lambda}\}$ is the spectral resolution of $\mathcal{L}$
and each $E_{\lambda}$ is a projection in $L^{2}(\mathbb{R}^{n})$
(see for instance, \cite{grigoryan2009heat}). The domain of $\phi(\mathcal{L})$
is given by 
\[
\mathrm{Dom}(\phi(\mathcal{L}))=\left\{ f\in L^{2}(\mathbb{R}^{n});\int_{0}^{\infty}|\phi(\lambda)|^{2}\,d\|E_{\lambda}f\|^{2}<\infty\right\} .
\]
The linear operator $\phi(\mathcal{L}):\mathrm{Dom}(\phi(\mathcal{L}))\rightarrow L^{2}(\mathbb{R}^{n})$
is understood, via Riesz representation theorem, in the following
sense, 
\[
\left\langle \phi(\mathcal{L})f,g\right\rangle :=\int_{0}^{\infty}\phi(\lambda)\,d\langle E_{\lambda}f,g\rangle,\quad f\in\mathrm{Dom}(\phi(\mathcal{L})),\ g\in L^{2}(\mathbb{R}^{n}),
\]
where $\langle\cdot,\cdot\rangle$ denotes the (real) inner product
in $L^{2}(\mathbb{R}^{n})$.

Now we are in a position to define the fractional operator $\mathcal{L}^{s}$.
Notice that $\lambda^{s}=\int_{0}^{\infty}(e^{-t\lambda}-1)t^{-1-s}dt/\Gamma(-s)$
for $s\in(0,1)$, where $\Gamma(-s):=-\Gamma(1-s)/s$, and $\Gamma$
is the Gamma function. We define, given $s\in(0,1)$, 
\begin{equation}
\mathcal{L}^{s}:=\int_{0}^{\infty}\lambda^{s}\,dE_{\lambda}=\frac{1}{\Gamma(-s)}\int_{0}^{\infty}\left(e^{-t\mathcal{L}}-\mbox{Id}\right)\,\frac{dt}{t^{1+s}},\label{eq:1111}
\end{equation}
where $e^{-t\mathcal{L}}$ given by 
\begin{equation}
e^{-t\mathcal{L}}:=\int_{0}^{\infty}e^{-t\lambda}\,dE_{\lambda}\label{eq:heat-semigroup}
\end{equation}
is a bounded self-adjoint operator in $L^{2}(\mathbb{R}^{n})$ for
each $t\ge0$. The operator family $\{e^{-t\mathcal{L}}\}_{t\ge0}$
is called the heat semigroup associated with $\mathcal{L}$ (cf. \cite{pazy2012semigroups}).
The domain of $\mathcal{L}^{s}$ is given by 
\begin{equation}
\mathrm{Dom}(\mathcal{L}^{s})=\left\{ f\in L^{2}(\mathbb{R}^{n});\int_{0}^{\infty}\lambda^{2s}\,d\|E_{\lambda}f\|^{2}<\infty\right\} .\label{eq:Domain}
\end{equation}
Notice for any $f\in\mathrm{Dom}(\mathcal{L}^{s})$ that, $\mathcal{L}^{s}f\in L^{2}(\mathbb{R}^{n})$
and is given, again in the sense of Riesz representation theorem,
by the formula 
\begin{equation}
\langle\mathcal{L}^{s}f,g\rangle=\frac{1}{\Gamma(-s)}\int_{0}^{\infty}\left\langle \left(e^{-t\mathcal{L}}f-f\right),g\right\rangle \frac{dt}{t^{1+s}},\quad g\in L^{2}(\mathbb{R}^{n}),\label{eq:heat representation for L^s}
\end{equation}
when $s\in(0,1)$.
\begin{rem}
We remark here that 
\begin{equation}
\mathrm{Dom}(\mathcal{L})\subseteq\mathrm{Dom}(\mathcal{L}^{s}),\quad s\in(0,1).\label{eq:DomL^s subset DomL}
\end{equation}
In fact, for any $f\in\mathrm{Dom}(\mathcal{L})\subseteq L^{2}(\mathbb{R}^{n})$,
one has 
\[
\begin{split}\int_{0}^{\infty}\lambda^{2s}\,d\|E_{\lambda}f\|^{2} & =\int_{1}^{\infty}\lambda^{2s}\,d\|E_{\lambda}f\|^{2}+\int_{0}^{1}\lambda^{2s}\,d\|E_{\lambda}f\|^{2}\\
 & \le\int_{0}^{\infty}\lambda^{2}\,d\|E_{\lambda}f\|^{2}+\int_{0}^{\infty}\,d\|E_{\lambda}f\|^{2}\\
 & =\|\mathcal{L}f\|_{L^{2}(\mathbb{R}^{n})}^{2}+\|f\|_{L^{2}(\mathbb{R}^{n})}^{2}<\infty.
\end{split}
\]

\end{rem}

\subsection{Sobolev Spaces}

For simplicity, we shall always consider real function spaces in this
paper. Our notations for Sobolev spaces are mainly followed by \cite{mclean2000strongly}.

Let $a\in\mathbb{R}$ be a constant. Let $H^{a}(\mathbb{R}^{n})=W^{a,2}(\mathbb{R}^{n})$
be the (fractional) Sobolev space endowed with the norm 
\[
\|u\|_{H^{a}(\mathbb{R}^{n})}:=\left\Vert \mathscr{F}^{-1}\big\{\left\langle \xi\right\rangle ^{a}\mathscr{F}u\big\}\right\Vert _{L^{2}(\mathbb{R}^{n})},
\]
where $\left\langle \xi\right\rangle =(1+|\xi|^{2})^{\frac{1}{2}}$.
It is known that for $s\in(0,1)$, $\|\cdot\|_{H^{s}(\mathbb{R}^{n})}$
has the following equivalent form 
\begin{equation}
\|u\|_{H^{s}(\mathbb{R}^{n})}:=\|u\|_{L^{2}(\mathbb{R}^{n})}+[u]_{H^{s}(\mathbb{R}^{n})}\label{eq:NormHs}
\end{equation}
where 
\[
[u]_{H^{s}(\mathcal{O})}^{2}:=\int_{\mathcal{O}\times\mathcal{O}}\frac{\left|u(x)-u(z)\right|^{2}}{|x-z|^{n+2s}}dxdz,
\]
for any open set $\mathcal{O}$ of $\mathbb{R}^{n}$.

Given any open set $\mathcal{O}$ of $\mathbb{R}^{n}$ and $a\in\mathbb{R}$,
we denote the following Sobolev spaces, 
\begin{align*}
H^{a}(\mathcal{O}) & :=\{u|_{\mathcal{O}};\,u\in H^{a}(\mathbb{R}^{n})\},\\
\widetilde{H}^{a}(\mathcal{O}) & :=\text{closure of \ensuremath{C_{c}^{\infty}(\mathcal{O})} in \ensuremath{H^{a}(\mathbb{R}^{n})}},\\
H_{0}^{a}(\mathcal{O}) & :=\text{closure of \ensuremath{C_{c}^{\infty}(\mathcal{O})} in \ensuremath{H^{a}(\mathcal{O})}},
\end{align*}
and 
\[
H_{\overline{\mathcal{O}}}^{a}:=\{u\in H^{a}(\mathbb{R}^{n});\,\mathrm{supp}(u)\subset\overline{\Omega}\}.
\]
The Sobolev space $H^{a}(\mathcal{O})$ is complete under the norm
\[
\|u\|_{H^{a}(\mathcal{O})}:=\inf\left\{ \|v\|_{H^{a}(\mathbb{R}^{n})};v\in H^{a}(\mathbb{R}^{n})\mbox{ and }v|_{\mathcal{O}}=u\right\} .
\]
It is known that $\widetilde{H}^{a}(\mathcal{O})\subseteq H_{0}^{a}(\mathcal{O})$,
and that $H_{\overline{\mathcal{O}}}^{a}$ is a closed subspace of
$H^{a}(\mathbb{R}^{n})$.
\begin{lem}
(\cite{mclean2000strongly}) Let $\Omega$ be a Lipschitz domain in
$\mathbb{R}^{n}$. Then

1. For any $a\in\mathbb{R}$, 
\[
\begin{split} & \widetilde{H}^{a}(\Omega)=H_{\overline{\Omega}}^{a}\subseteq H_{0}^{a}(\Omega),\\
 & \left(H^{a}(\Omega)\right)^{*}=\widetilde{H}^{-a}(\Omega)\mbox{ and }\left(\widetilde{H}^{a}(\Omega)\right)^{*}=H^{-a}(\Omega).
\end{split}
\]

2. For $a\ge0$ and $a\notin\{\frac{1}{2},\frac{3}{2},\frac{5}{2},\frac{7}{2},\ldots\}$,
\[
\widetilde{H}^{a}(\Omega)=H_{0}^{a}(\Omega).
\]

\end{lem}

\subsection{The Operator $\mathcal{L}^{s}$}

\label{sec:LsHsHeatkernel} In this paper, we consider $\mathcal{L}$
to be a linear second order partial differential operator of the divergence
form 
\begin{equation}
\mathcal{L}:=-\nabla\cdot(A(x)\nabla).\label{eq: Local elliptic operator-1}
\end{equation}
We assume that the $n$-by-$n$ matrix $A(x)=(a_{ij}(x))_{i,j=1}^{n}$
is symmetric and that $\mathcal{L}$ is uniformly elliptic, namely,
\begin{equation}
\Lambda^{-1}|\xi|^{2}\leq\xi^{T}A(x)\xi\leq\Lambda|\xi|^{2}\quad\mbox{for all }x,\xi\in\mathbb{R}^{n},
\end{equation}
for some positive constant $\Lambda$. We also assume that the variable
coefficients of $\mathcal{L}$ are smooth, i.e., 
\begin{equation}
a_{ij}=a_{ji}\in C^{\infty},\quad1\le i,j\le n.\label{eq:Asmoothsymmetric}
\end{equation}

It is easy to see that the operator $\mathcal{L}$ introduced in \eqref{eq: Local elliptic operator-1}-\eqref{eq:Asmoothsymmetric}
is well-defined on $C_{0}^{\infty}(\mathbb{R}^{n})$, which is dense
in the Hilbert space $L^{2}(\mathbb{R}^{n})$. However, $\mathcal{L}$
is not self-adjoint in the domain $C_{0}^{\infty}(\mathbb{R}^{n})$.
In fact, one can verify in this case that, the adjoint operator admits
the domain $\mathrm{Dom}(\mathcal{L}^{*})=\{f\in L^{2}(\mathbb{R}^{n});\,\mathcal{L}f\in L^{2}(\mathbb{R}^{n})\}$,
which does not coincide with $C_{0}^{\infty}(\mathbb{R}^{n})$. In
order to define the fractional power $\mathcal{L}^{s}$ by applying
the spectral theory we briefly sketched in Section~\ref{sec:SpectralTheory},
one needs firstly to extend $\mathcal{L}$ as a self-adjoint operator
densely defined in $L^{2}(\mathbb{R}^{n})$.

It is known, see for instance \cite{grigoryan2009heat}, that $\mathcal{L}$
with the domain 
\begin{equation}
\mathrm{Dom}(\mathcal{L})=H^{2}(\mathbb{R}^{n})\label{eq:DomL}
\end{equation}
is the maximal extension such that $\mathcal{L}$ is self-adjoint
and densely defined in $L^{2}(\mathbb{R}^{n})$. Moreover, it is natural
to expect that $\mathrm{Dom}(\mathcal{L}^{s})$ is close to the Sobolev
space $H^{2s}(\mathbb{R}^{n})$, which is shown, at least when $s=1/2$,
that $\mathrm{Dom}(\mathcal{L}^{s})=H^{2s}(\mathbb{R}^{n})$ (cf.
\cite{davies1990heat,grigoryan2009heat}). Next, we would like to
extend the definition of $\mathcal{L}^{s}$ from its domain $\mathrm{Dom}(\mathcal{L}^{s})$
introduced in \eqref{eq:Domain} to $H^{s}(\mathbb{R})$, using heat
kernels and theirs estimates.

It is known that for $\mathcal{L}$ satisfying \eqref{eq: Local elliptic operator-1}-\eqref{eq:Asmoothsymmetric},
the bounded operator $e^{-t\mathcal{L}}$ given in \eqref{eq:heat-semigroup}
admits a symmetric (heat) kernel $p_{t}(x,z)$ (cf. \cite{grigoryan2009heat}).
In other words, one has for any $t\in\mathbb{R}_{+}:=(0,\infty)$
and any $f\in L^{2}(\mathbb{R}^{n})$ that 
\begin{equation}
\left(e^{-t\mathcal{L}}f\right)(x)=\int_{\mathbb{R}^{n}}p_{t}(x,z)f(z)\,dz,\quad x\in\mathbb{R}^{n}.
\end{equation}
Moreover for any $t\in\mathbb{R}_{+}$, the kernel $p_{t}(\cdot,\cdot)$
is symmetric and admits the following estimates (cf. \cite{davies1990heat})
with some positive constants $c_{j}$ and $b_{j}$, $j=1,2$, 
\begin{equation}
c_{1}e^{-b_{1}\frac{|x-z|^{2}}{t}}t^{-\frac{n}{2}}\leq p_{t}(x,z)\leq c_{2}e^{-b_{2}\frac{|x-z|^{2}}{t}}t^{-\frac{n}{2}},\quad x,z\in\mathbb{R}^{n}.\label{eq:pointwise estimates for p_t}
\end{equation}
By applying similar arguments as in the proof of \cite[Theorem 2.4]{caffarelli2016fractional},
one has for $f,g\in\mathrm{Dom}(\mathcal{L}^{s})$ that 
\begin{equation}
\langle\mathcal{L}^{s}f,g\rangle=\frac{1}{2\Gamma(-s)}\int_{0}^{\infty}\int_{\mathbb{R}^{n}\times\mathbb{R}^{n}}(f(x)-f(z))(g(x)-g(z))p_{t}(x,z)dxdz\frac{dt}{t^{1+s}},\label{eq:LsIntegral0}
\end{equation}

Now we define 
\begin{equation}
\mathcal{K}_{s}(x,z):=\frac{1}{\Gamma(-s)}\int_{0}^{\infty}p_{t}(x,z)\frac{dt}{t^{1+s}}.\label{eq:kernel}
\end{equation}
It is seen from \eqref{eq:pointwise estimates for p_t} that $\mathcal{K}_{s}$
enjoys the following pointwise estimate 
\begin{equation}
\frac{C_{1}}{|x-z|^{n+2s}}\leq\mathcal{K}_{s}(x,z)=\mathcal{K}_{s}(z,x)\leq\dfrac{C_{2}}{|x-z|^{n+2s}},\quad x,z\in\mathbb{R}^{n},\label{eq:pointwise estimate for kernel K}
\end{equation}
with some positive constants $C_{1}$, $C_{2}$. Hence it is obtain
by recalling the norm \eqref{eq:NormHs} of $H^{s}(\mathbb{R}^{n})$
that for any $f,g\in H^{s}(\mathbb{R}^{n})$, the right hand side
(RHS) of \eqref{eq:LsIntegral0} coincides with 
\[
\frac{1}{2}\int_{\mathbb{R}^{n}\times\mathbb{R}^{n}}(f(x)-f(z))(g(x)-g(z))\mathcal{K}_{s}(x,z)dxdz.
\]
Therefore, it is natural to extend the definition of $\mathcal{L}^{s}$
from $\mathrm{Dom}(\mathcal{L}^{s})$ to $H^{s}(\mathbb{R}^{n})$
in the following distributional sense 
\begin{equation}
\langle\mathcal{L}^{s}f,g\rangle:=\frac{1}{2}\int_{\mathbb{R}^{n}\times\mathbb{R}^{n}}(f(x)-f(z))(g(x)-g(z))\mathcal{K}_{s}(x,z)dxdz.\label{eq:integral represent for nonlocal}
\end{equation}
Moreover, it is obtained from \eqref{eq:pointwise estimate for kernel K}
that, there exists a positive constant $C$ such that the operator
$\mathcal{L}^{s}$ defined in \eqref{eq:integral represent for nonlocal}
satisfies 
\begin{equation}
\left|\langle\mathcal{L}^{s}f,g\rangle\right|\le C\|u\|_{H^{s}(\mathbb{R}^{n})}\|v\|_{H^{s}(\mathbb{R}^{n})},\quad u,v\in H^{s}(\mathbb{R}^{n}).\label{eq:LsBoundedHs}
\end{equation}
Thus, the definition \eqref{eq:integral represent for nonlocal} gives
a bounded linear operator 
\[
\mathcal{L}^{s}:H^{s}(\mathbb{R}^{n})\longrightarrow H^{-s}(\mathbb{R}^{n}).
\]
We observe by using the symmetry $\mathcal{K}_{s}(x,z)=\mathcal{K}_{s}(z,x)$
that $\mathcal{L}^{s}$ is also symmetric, namely 
\begin{equation}
\langle\mathcal{L}^{s}f,g\rangle=\langle\mathcal{L}^{s}g,f\rangle,\quad f,g\in H^{s}(\mathbb{R}^{n}).\label{eq:LsSymmetric}
\end{equation}
Furthermore, it is obtained that 
\begin{align*}
\left\langle \mathcal{L}^{s}f,g\right\rangle = & \frac{1}{2}\lim_{\epsilon\to0^{+}}\int_{\mathbb{R}^{n}}\int_{|x-z|>\epsilon}(f(x)-f(z))(g(x)-g(z))\mathcal{K}_{s}(x,z)dxdz\\
= & \frac{1}{2}\lim_{\epsilon\to0^{+}}\int_{\mathbb{R}^{n}}\int_{|x-z|>\epsilon}(f(x)-f(z))g(x)\mathcal{K}_{s}(x,z)dxdz\\
 & +\frac{1}{2}\lim_{\epsilon\to0^{+}}\int_{\mathbb{R}^{n}}\int_{|x-z|>\epsilon}(f(x)-f(z))g(x)\mathcal{K}_{s}(x,z)dzdx\\
= & \int_{\mathbb{R}^{n}}g(x)\lim_{\epsilon\to0^{+}}\int_{|x-z|>\epsilon}(f(x)-f(z))\mathcal{K}_{s}(x,z)dzdx,
\end{align*}
holds for all $f,g\in H^{s}(\mathbb{R}^{n})$. Hence one can also
write 
\begin{equation}
\begin{split}\left(\mathcal{L}^{s}f\right)(x)= & \lim_{\epsilon\to0^{+}}\int_{|x-z|>\epsilon}(f(x)-f(z))\mathcal{K}_{s}(x,z)dz,\quad f\in H^{s}(\mathbb{R}^{n}).\end{split}
\label{eq:PV for L^s}
\end{equation}

\section{Dirichlet problems for $\mathcal{L}^{s}+q$\label{Section 3}}

In a continuation to the general case, we proceed our discussions
by introducing the state spaces followed by the Dirichlet problem
and associated DN map for for $\mathcal{L}^{s}+q$.

\subsection{Well-Posedness}

Throughout this section, we shall always let $\Omega\subseteq\mathbb{R}^{n}$
be a bounded Lipschitz domain, $q$ be a potential in $L^{\infty}(\Omega)$
and $s\in(0,1)$ be a constant. We consider the following nonlocal
Dirichlet problem for the nonlocal operator $\mathcal{L}^{s}$, 
\begin{equation}
\begin{cases}
(\mathcal{L}^{s}+q)u=f & \mbox{ in \ensuremath{\Omega}},\\
u=g & \mbox{ in }\Omega_{e}.
\end{cases}\label{eq:Nonlocal Dirichlet problem}
\end{equation}
Define the bilinear form $B_{q}(\cdot,\cdot)$ by 
\begin{equation}
\mathcal{B}_{q}(v,w):=\left\langle \mathcal{L}^{s}v,w\right\rangle +\int_{\Omega}q(x)v(x)w(x)\,dx,\quad v,w\in H^{s}(\mathbb{R}^{n})
\end{equation}
with $\mathcal{L}^{s}$ given by the form \eqref{eq:integral represent for nonlocal}.
It is seen from \eqref{eq:LsSymmetric} that $\mathcal{B}_{q}$ is
symmetric, and from \eqref{eq:LsBoundedHs} that $\mathcal{B}_{q}$
is a bounded in $H^{s}(\mathbb{R}^{n})\times H^{s}(\mathbb{R}^{n})$,
i.e., 
\begin{equation}
\left|\mathcal{B}_{q}(v,w)\right|\leq C\|v\|_{H^{s}(\mathbb{R}^{n})}\|w\|_{H^{s}(\mathbb{R}^{n})},\quad v,w\in H^{s}(\mathbb{R}^{n}).\label{eq:Upper bound of bilinear}
\end{equation}
It is further obtained that $\mathcal{B}_{q}$ can be also regarded
as a symmetric bounded bilinear form in the space $\widetilde{H}^{s}(\Omega)$.
In fact, by using \eqref{eq:Upper bound of bilinear} and the fact
that $C_{0}^{\infty}(\Omega)$ is dense in $\widetilde{H}^{s}(\Omega)$,
one can define for any $v\in\widetilde{H}^{s}(\Omega)$ that, 
\begin{equation}
\mathcal{B}_{q}(v,\phi)=\langle\mathcal{L}^{s}\tilde{v},\phi\rangle+\int_{\Omega}q(x)v(x)\phi(x)\,dx,\quad\phi\in C_{0}^{\infty}(\Omega),\label{eq:bilinear form1}
\end{equation}
where $\tilde{v}\in H^{s}(\mathbb{R}^{n})$ is an extension of $v$
such that $\tilde{v}|_{\Omega}=v$. 
\begin{defn}
Let $\Omega$ be a bounded Lipschitz domain in $\mathbb{R}^{n}$.
Given $f\in H^{-s}(\Omega)$ and $g\in H^{s}(\mathbb{R}^{n})$, we
say that $u\in H^{s}(\mathbb{R}^{n})$ is a (weak) solution of \eqref{eq:Nonlocal Dirichlet problem}
if $\widetilde{u}_{g}:=u-g\in\widetilde{H}^{s}(\Omega)$ and 
\begin{equation}
\mathcal{B}_{q}(u,\phi)=\left\langle f,\phi\right\rangle \quad\mbox{for any \ensuremath{\phi\in C_{0}^{\infty}(\Omega)}},\label{eq:weak solution via bilinear}
\end{equation}
or equivalently 
\begin{equation}
\mathcal{B}_{q}(\widetilde{u}_{g},\phi)=\left\langle f-(\mathcal{L}^{s}+q)g,\phi\right\rangle \quad\mbox{for any \ensuremath{\phi\in C_{0}^{\infty}(\Omega)}}.\label{eq:weak solution via bilinear1}
\end{equation}
\end{defn}
\begin{rem}
It is noticed that the space $C_{0}^{\infty}(\Omega)$ of test functions
in \eqref{eq:weak solution via bilinear} and \eqref{eq:weak solution via bilinear1}
can be replaced by $\widetilde{H}^{s}(\Omega)$. 
\end{rem}
The well-posedness of the Dirichlet problem \eqref{eq:Nonlocal Dirichlet problem}
is shown by the following more general result. 
\begin{prop}
\label{props:wellposedness} Let $\Omega$ be a bounded Lipschitz
domain in $\mathbb{R}^{n}$ and $q\in L^{\infty}(\Omega)$. The following
results hold.

(1) There is a countable set $\Sigma=\{\lambda_{j}\}_{j=1}^{\infty}$
of real numbers $\lambda_{1}\le\lambda_{2}\le\ldots\rightarrow\infty$,
such that given $\lambda\in\mathbb{R}\setminus\Sigma$, for any $f\in H^{-s}(\Omega)$
and any $g\in H^{s}(\mathbb{R}^{n})$, there is a unique $u\in H^{s}(\mathbb{R}^{n})$
satisfying $u-g\in\widetilde{H}^{s}(\Omega)$ and 
\begin{equation}
\mathcal{B}_{q}(u,v)-\lambda(u,v)_{L^{2}}=\left\langle f,v\right\rangle \quad\mbox{for any \ensuremath{v\in\widetilde{H}^{s}(\Omega)}}.\label{eq:forwardEigen}
\end{equation}
Moreover, 
\begin{equation}
\|u\|_{H^{s}(\mathbb{R}^{n})}\leq C_{0}\left(\|f\|_{H^{-s}(\Omega)}
+\|g\|_{H^{s}(\mathbb{R}^{n})}\right),\label{eq:well posed estimate}
\end{equation}
for some constant $C_{0}>0$ independent of $f$ and $g$.

(2) The condition \eqref{eq:eigenvalue condition} holds if and only
if $0\notin\Sigma$.

(3) If $q\ge0$ a.e. in $\Omega$, then $\Sigma\subseteq\mathbb{R}_{+}$,
and hence \eqref{eq:eigenvalue condition} always holds. \end{prop}
\begin{proof}
It is obtained from \eqref{eq:pointwise estimate for kernel K} and
\eqref{eq:integral represent for nonlocal} that 
\begin{equation}
\langle\mathcal{L}^{s}v,v\rangle=\frac{1}{2}\int_{\mathbb{R}^{n}\times\mathbb{R}^{n}}\left|v(x)-v(z)\right|^{2}\mathcal{K}_{s}(x,z)dxdz\ge c_{0}\|v\|_{H^{s}(\mathbb{R}^{n})},
\end{equation}
for any $v\in H^{s}(\mathbb{R}^{n})$ with some constant $c_{0}>0$
independent of $v$. As a consequence, 
\begin{equation}
\mathcal{B}_{q}(v,v)+\lambda_{0}(v,v)_{L^{2}}\ge c_{0}\|v\|_{\widetilde{H}^{s}(\Omega)}^{2},
\end{equation}
for any $v\in\widetilde{H}^{s}(\Omega)$, where $\lambda_{0}$ is
a constant such that $\lambda_{0}\leq\|q_{-}\|_{L^{\infty}(\Omega)}$
with $q_{-}(x):=-\min\{0,q(x)\}$. On the other hand, it is easy to
see from \eqref{eq:Upper bound of bilinear} that 
\begin{equation}
\left|\mathcal{B}_{q}(w,v)+\lambda_{0}(w,v)_{L^{2}}\right|\le(C+\lambda_{0})\|w\|_{\widetilde{H}^{s}(\Omega)}\|v\|_{\widetilde{H}^{s}(\Omega)},
\end{equation}
holds for any $w,v\in\widetilde{H}^{s}(\Omega)$. Hence, we know that
the bilinear form $\mathcal{B}_{q}(\cdot,\cdot)+\lambda_{0}(\cdot,\cdot)_{L^{2}}$
is bounded and coercive. Therefore, given any $f\in H^{-s}(\Omega)=\left(\widetilde{H}^{s}(\Omega)\right)^{*}$,
there is a unique $u\in\widetilde{H}^{s}(\Omega)$ such that 
\begin{equation}
\mathcal{B}_{q}(u,v)+\lambda_{0}(u,v)_{L^{2}}=\left\langle f,v\right\rangle \quad\mbox{for any \ensuremath{v\in\widetilde{H}^{s}(\Omega)}},\label{eq:wellposeProof}
\end{equation}
and that 
\begin{equation}
\|u\|_{\widetilde{H}^{s}(\Omega)}\le C\|f\|_{H^{-s}(\Omega)}
\end{equation}
with some constant $C$ independent of $f$. Denote by $\mathcal{G}_{0}$
the operator mapping $f$ to the solution $u$ of \eqref{eq:wellposeProof}.
Then $\mathcal{G}_{0}$ is bounded from $H^{-s}(\Omega)$ to $\widetilde{H}^{s}(\Omega)$
with a bounded inverse.

Now, suppose $\widetilde{u}_{g}\in\widetilde{H}^{s}(\Omega)$ satisfying
\eqref{eq:forwardEigen} with $u=\widetilde{u}_{g}$. Then one has
\[
\widetilde{u}_{g}=\mathcal{G}_{0}\left(f+(\lambda+\lambda_{0})\tilde{u}_{g}\right),
\]
which implies 
\begin{equation}
\left(\frac{1}{\lambda+\lambda_{0}}\mathrm{Id}-\mathcal{G}_{0}\right)\widetilde{u}_{g}=\mathcal{G}_{0}f,\label{eq:proofWellposedness}
\end{equation}
where $\mathrm{Id}$ denotes the identity map in $\widetilde{H}^{s}(\Omega)$.
By compact Sobolev embedding, it is observed that $\mathcal{G}_{0}$
is compact in $\widetilde{H}^{s}(\Omega)$. Thus by the spectral properties
of compact operators, $\mathcal{G}_{0}^{-1}$ has discrete spectrum
$\{\frac{1}{\lambda_{j}+\lambda_{0}}\}_{j=1}^{\infty}$ consisting
only eigenvalues with $\lambda_{j}\rightarrow\infty$ as $j$ increases.
Denote $\Sigma=\{\lambda_{j}\}_{j=1}^{\infty}$. Then by the Fredholm
alternative one has for any $\lambda\notin\Sigma$, the operator 
\[
\left(\frac{1}{\lambda+\lambda_{0}}\mathrm{Id}-\mathcal{G}_{0}\right):\widetilde{H}^{s}(\Omega)\rightarrow\widetilde{H}^{s}(\Omega),
\]
is injective and has a bounded inverse. Therefore, the equation \eqref{eq:proofWellposedness}
is uniquely solvable, providing $\lambda\notin\Sigma$, with the following
estimate of the solution $\widetilde{u}_{g}$, 
\begin{equation}
\|\widetilde{u}_{g}\|_{\widetilde{H}^{s}(\Omega)}\le C\|f\|_{H^{-s}(\Omega)},
\end{equation}
for some constant $C>0$ independent of $\widetilde{u}_{g}$ and $f$.

The rest of the proof for the statement (1) is completed by considering
$\widetilde{u}_{g}=u-g$. The result in (2) is a direct consequence
of (1). Finally, by taking $\lambda_{0}=0$ in the previous arguments,
one already sees (3). 
\end{proof}
Next we consider the Dirichlet problem \eqref{eq:Nonlocal Dirichlet problem}
with a zero RHS, namely, 
\begin{equation}
\begin{cases}
(\mathcal{L}^{s}+q)\,u=0 & \mbox{ in \ensuremath{\Omega}},\\
u=g & \mbox{ in }\Omega_{e}.
\end{cases}\label{eq:Nonlocal Dirichlet problemHomo}
\end{equation}
In the rest of the paper, we shall always assume that $q\in L^{\infty}(\Omega)$
satisfies \eqref{eq:eigenvalue condition}, or equivalently, $0\notin\Sigma$
with the set $\Sigma$ given in Proposition~\ref{props:wellposedness}.
Under this assumption, it is shown in Proposition~\ref{props:wellposedness}
that given any $g\in H^{s}(\mathbb{R}^{n})$, the Dirichlet problem
\eqref{eq:Nonlocal Dirichlet problemHomo} admits a unique solution
$u\in H^{s}(\mathbb{R}^{n})$ such that 
\begin{equation}
\|u\|_{H^{s}(\mathbb{R}^{n})}\leq C\|g\|_{H^{s}(\mathbb{R}^{n})}.\label{eq:well posed estimateHomo}
\end{equation}
Recall that $u\in H^{s}(\mathbb{R}^{n})$ is called a solution of
\eqref{eq:Nonlocal Dirichlet problemHomo} if $u-g\in\widetilde{H}^{s}(\Omega)$
and $\mathcal{B}_{q}(u,v)=0$ for any $v\in\widetilde{H}^{s}(\Omega)$.
\begin{rem}
\label{rem:uIndependentg} We emphasize here that the solution $u$
of \eqref{eq:Nonlocal Dirichlet problemHomo} does not depends on
the value of $g$ in $\Omega$. To be more precise, let $g_{1},g_{2}\in H^{s}(\mathbb{R}^{n})$
be such that $g_{1}-g_{2}\in\widetilde{H}^{s}(\Omega)=H_{\overline{\Omega}}^{s}$.
Denote $u_{j}\in H^{s}(\mathbb{R}^{n})$ as the solution of \eqref{eq:Nonlocal Dirichlet problemHomo}
with the Dirichlet data $g_{j}$ for each $j=1,2$. It is observed
that 
\[
\tilde{u}:=u_{1}-u_{2}=(u_{1}-g_{1})-(u_{2}-g_{2})+(g_{1}-g_{2})\in\widetilde{H}^{s}(\Omega)
\]
and $\mathcal{B}_{q}(\tilde{u},v)=0$ for any $v\in\widetilde{H}^{s}(\Omega)$.
Thus by the unique solvability of \eqref{eq:Nonlocal Dirichlet problemHomo}
with $g=0$ one has $\tilde{u}=0$. Therefore, one can actually consider
the nonlocal problem \eqref{eq:Nonlocal Dirichlet problemHomo} with
Dirichlet data in the quotient space 
\begin{equation}
X:=H^{s}(\mathbb{R}^{n})/H_{\overline{\Omega}}^{s}\cong H^{s}(\Omega_{e}),\label{eq:Xquotient}
\end{equation}
provided that $\Omega$ is Lipschitz. 
\end{rem}

\subsection{The DN Map}

We define in this section the associated DN map for $\mathcal{L}^{s}+q$
via the bilinear form $\mathcal{B}_{q}$ in \eqref{eq:weak solution via bilinear}. 
\begin{prop}
\label{prop:DNmap} (DN map) Let $\Omega$ be a bounded Lipschitz
domain in $\mathbb{R}^{n}$ for $n\geq2$, $s\in(0,1)$ and $q\in L^{\infty}(\Omega)$
satisfy the eigenvalue condition \eqref{eq:eigenvalue condition}.
Let $X$ be the quotient space given in \eqref{eq:Xquotient}. Define
\begin{equation}
\left\langle \Lambda_{q}[g],[h]\right\rangle :=\mathcal{B}_{q}(u,h),\quad[g],[h]\in X,\label{eq:equvalent integration by parts}
\end{equation}
where $g,h\in H^{s}(\mathbb{R}^{n})$ are representatives of the classes
$[g],[h]\in X$ respectively, and $u\in H^{s}(\mathbb{R}^{n})$ is
the solution of \eqref{eq:Nonlocal Dirichlet problemHomo} with the
Dirichlet data $g$. Then, 
\[
\Lambda_{q}:X\to X^{*},
\]
which is bounded. Moreover, we have the following symmetry property
for $\Lambda_{q}$, 
\begin{equation}
\left\langle \Lambda_{q}[g],[h]\right\rangle =\left\langle \Lambda_{q}[h],[g]\right\rangle ,\quad[g],[h]\in X.\label{eq:adjoint operator}
\end{equation}
\end{prop}
\begin{proof}
We first show that $\Lambda_{q}$ given in \eqref{eq:equvalent integration by parts}
is well-defined. Recall from Remark~\ref{rem:uIndependentg} that,
the solution to \eqref{eq:Nonlocal Dirichlet problemHomo} with Dirichlet
data $\tilde{g}\in H^{s}(\mathbb{R}^{n})$ is the same as the solution
with data $g$, as long as $\tilde{g}-g\in\widetilde{H}^{s}(\Omega)$.
Thus the RHS of \eqref{eq:equvalent integration by parts} is invariant
under the different choices of the representative $g\in H^{s}(\mathbb{R}^{n})$
for $[g]\in X$. In addition, one has 
\[
\mathcal{B}_{q}(u,\tilde{h})=\mathcal{B}_{q}(u,h)+\mathcal{B}_{q}(u,\tilde{h}-h)=\mathcal{B}_{q}(u,h),
\]
for any $\tilde{h}\in H^{s}(\mathbb{R}^{n})$ such that $\tilde{h}-h\in\widetilde{H}^{s}(\Omega)$.
Therefore, the RHS of \eqref{eq:equvalent integration by parts} is
well determined by $[g],[h]\in X$.

From the boundedness \eqref{eq:Upper bound of bilinear} of $\mathcal{B}_{q}$,
one has 
\begin{align*}
\left|\left\langle \Lambda_{q}[g],[h]\right\rangle \right| & \leq C\|g_{0}\|_{H^{s}(\mathbb{R}^{n})}\|h_{0}\|_{H^{s}(\mathbb{R}^{n})}\\
 & \leq C_{1}\|[g]\|_{X}\|[h]\|_{X},
\end{align*}
by properly choosing representatives $g_{0},h_{0}\in H^{s}(\mathbb{R}^{n})$
for $[g],[h]\in X$. The symmetry of $X$ is a direct consequence
of the symmetry of the bilinear form $\mathcal{B}_{q}$.

The proof is completed. 
\end{proof}
Recall that the quotient space $X$ is isometric to $H^{s}(\Omega_{e})$,
since $\Omega\subseteq\mathbb{R}^{n}$ is a Lipschitz domain. Hence
one can always regard the operator $\Lambda_{q}$ defined in Proposition~\ref{prop:DNmap}
as 
\[
\Lambda_{q}:H^{s}(\Omega_{e})\to\left(H^{s}(\Omega_{e})\right)^{*}=H_{\overline{\Omega_{e}}}^{-s}=\widetilde{H}^{-s}(\Omega_{e}).
\]

In general, for any $\widetilde{h}\in H^{s}(\mathbb{R}^{n})$ we have
\begin{align}
(\Lambda_{q}[g],[h])_{X^{*}\times X} & =\mathcal{B}_{q}(u_{g},\widetilde{h})\nonumber \\
 & =\int_{\mathbb{R}^{n}}\widetilde{h}\mathcal{L}^{s}u_{g}dx+\int_{\Omega}qu_{g}\widetilde{h}dx\nonumber \\
 & =\int_{\Omega_{e}}\widetilde{h}\mathcal{L}^{s}u_{g}dx\nonumber \\
 & =\int_{\Omega_{e}}h\mathcal{L}^{s}u_{g}\,dx.\label{t15}
\end{align}
We must note that, the above integral becomes zero whenever $h\in\widetilde{H}^{s}(\Omega_{e})$,
(i.e. whenever $h$ has support only in $\Omega_{e}$): 
\[
(\Lambda_{q}[g],[h])_{X^{*}\times X}=0,\mbox{ for any }h\in\widetilde{H}^{s}(\Omega_{e}).
\]
Then from \eqref{t15} we have 
\[
(\Lambda_{q}[g],[h])_{H_{\overline{\Omega}_{e}}^{-s}(\mathbb{R}^{n})\times H^{s}(\Omega_{e})}=\int_{\Omega_{e}}h\mathcal{L}^{s}u_{g}\,dx,\mbox{ for any }h\in H^{s}(\Omega_{e}).
\]
This implies that 
\begin{equation}
\Lambda_{q}[g]=\left.\mathcal{L}^{s}u_{g}\right|_{\Omega_{e}}.\label{t16}
\end{equation}
Let us continue to give another representation of $\Lambda_{q}[g]$
involving the Neumann operator $\mathcal{N}_{s}$. We introduce the
anisotropic nonlocal Neumann operator $\mathcal{N}_{s}$ analogues
to the Neumann operator which is initiated in \cite{dipierro2014nonlocal}
for the fractional Laplacian operator $(-\Delta)^{s}$. Here we define
the anisotropic nonlocal Neumann operator $\mathcal{N}_{s}$ for $\mathcal{L}^{s}$
over the exterior domain $\Omega_{e}$ as follows: 
\begin{equation}
\mathcal{N}_{s}u(x):=\int_{\Omega}\mathcal{K}_{s}(x,z)(u(x)-u(z))\,dz,\mbox{ for }x\in\Omega_{e}\mbox{ and }u\in H^{s}(\mathbb{R}^{n})\label{Neumann}
\end{equation}
where $\mathcal{K}_{s}(x,z)$ (cf. \eqref{eq:kernel}) is the kernel
of $\mathcal{L}$ introduced in \eqref{eq:PV for L^s}. 
\begin{lem}
Let $\Omega\subseteq\mathbb{R}^{n}$ as mentioned above. Then 
\begin{equation}
\Lambda_{q}[g]=\left.\left(\mathcal{N}_{s}u_{g}-mg+\mathcal{L}^{s}(E_{0}g)\right)\right|_{\Omega_{e}}.\label{t17}
\end{equation}
where $m\in C^{\infty}(\Omega_{e})$ is given by $m(x)=\int_{\Omega}\mathcal{K}_{s}(x,z)\,dz$
and $E_{0}$ is extension by zero, i.e. $E_{0}g=\chi_{\Omega_{e}}g$.\end{lem}
\begin{proof}
Since $\Omega$ is a Lipschitz domain, from \eqref{t16} we have:
\[
\Lambda_{q}[g]=\left(\mathcal{L}^{s}u_{g}\right)|_{\Omega_{e}}=\left.\left(\mathcal{L}^{s}(\chi_{\Omega}u_{g})+\mathcal{L}^{s}(\chi_{\Omega_{e}}u_{g})\right)\right|_{\Omega_{e}},
\]
as we know if $g\in H^{s}(\Omega_{e})$, then $g\in H^{\alpha}(\Omega_{e})$
for some $\alpha\in(-1/2,1/2)$ and hence $E_{0}g,u_{g}\in H^{\alpha}(\mathbb{R}^{n})$.
Recall also that $\chi_{\Omega}$ and $(1-\chi_{\Omega})$ are pointwise
multipliers on $H^{\alpha}(\mathbb{R}^{n})$ (see \cite{ghosh2016calder}).
Now from the pointwise definition of $\mathcal{L}^{s}$ given in \eqref{eq:PV for L^s},
and the Neumann operator in \eqref{Neumann} it simply follows that
: 
\[
\left(\mathcal{L}^{s}(\chi_{\Omega}u_{g})\right)|_{\Omega_{e}}=\left(\mathcal{N}_{s}u_{g}-mg\right)|_{\Omega_{e}}
\]
where $m\in C^{\infty}(\Omega_{e})$ is given by $m(x)=\int_{\Omega}\mathcal{K}_{s}(x,z)\,dz$. 
\end{proof}
Hence, we have two representation of $\Lambda_{q}[g]$ are given by
\eqref{t16} and \eqref{t17}. 
\begin{rem}
\label{smooth DN map Lemma}Let $\Omega\subseteq\mathbb{R}^{n}$ be
a bounded open set with $C^{\infty}$-smooth boundary. Suppose the
matrix $A(x)$ given in \eqref{eq: Local elliptic operator-1} satisfying
\eqref{eq:ellipticity and symmetry condition}, potential $q(x)$
and the source term $g(x)$ are $C^{\infty}$-smooth functions in
$\mathbb{R}^{n}$, $\Omega$, and $\Omega_{e}$ , respectively. Then
for any $\beta\geq0$ with $s-\dfrac{1}{2}<\beta<\dfrac{1}{2}$, the
DN map is given by 
\[
\Lambda_{q}:H^{s+\beta}(\Omega_{e})\to H^{-s+\beta}(\Omega_{e}),\ \ \Lambda_{q}g=\mathcal{L}^{s}u_{g}|_{\Omega_{e}},
\]
where $u_{g}\in H^{s+\beta}(\mathbb{R}^{n})$ is a solution of $(\mathcal{L}^{s}+q)u=0$
in $\Omega$ with $u=g$ in $\Omega_{e}$. \end{rem}
\begin{proof}
Since $A(x)\in C^{\infty}(\mathbb{R}^{n})$ and $\mathcal{L}^{s}$
is the fractional operator with $C^{\infty}$-smooth coefficients
of order $2s$, it bounds to satisfy the $s$-transmission eigenvalue
condition given in \cite{grubb2015fractional}. Then the proof becomes
analogues to the proof of \cite[Lemma 3.1]{ghosh2016calder} and we
omit here. 
\end{proof}
We end this section by deriving couple of results regarding the integral
identity in our case. 
\begin{lem}
\label{lem:(Integral-identity)-Let}(Integral identity) Let $\Omega\subseteq\mathbb{R}^{n}$
as mentioned above, $s\in(0,1)$ and $q_{1},q_{2}\in L^{\infty}(\Omega)$
satisfy \eqref{eq:eigenvalue condition}. For any $g_{1},g_{2}\in H^{s}(\Omega_{e})$
one has 
\begin{equation}
((\Lambda_{q_{1}}-\Lambda_{q_{2}})[g_{1}],[g_{2}])=((q_{1}-q_{2})r_{\Omega}u_{1},r_{\Omega}u_{2})_{\mathbb{R}^{n}}\label{eq:integral identity}
\end{equation}
where $u_{j}\in H^{s}(\mathbb{R}^{n})$ solves $(\mathcal{L}^{s}+q_{j})u_{j}=0$
in $\Omega$ with $u_{j}|_{\Omega_{e}}=g_{j}$ for $j=1,2$.\end{lem}
\begin{proof}
By \eqref{eq:adjoint operator}, we have 
\begin{align*}
((\Lambda_{q_{1}}-\Lambda_{q_{2}})[g_{1}],[g_{2}]) & =(\Lambda_{q_{1}}[g_{1}],[g_{2}])-([g_{1}],\Lambda_{q_{2}}[g_{2}])\\
 & =B_{q_{1}}(u_{1},u_{2})-B_{q_{2}}(u_{1},u_{2})\\
 & =((q_{1}-q_{2})r_{\Omega}u_{1},r_{\Omega}u_{2})_{\mathbb{R}^{n}}.
\end{align*}

\end{proof}

\section{Extension Problems for $\mathcal{L}^{s}$\label{Section 4}}

In this section, we introduce an extension problem, which characterize
the nonlocal operator $\mathcal{L}^{s}$. For convenience, we introduce
the following notations.

\subsection*{Notations in $\mathbb{R}^{n+1}$}

We shall always, unless otherwise specified, refer the notation $(x,y)\in\mathbb{R}^{n+1}$
with $x\in\mathbb{R}^{n}$ and $y\in\mathbb{R}$. Let $\mathbb{R}_{+}^{n+1}$
be the (open) upper half space of $\mathbb{R}^{n+1}$, namely, $\mathbb{R}_{+}^{n+1}:=\left\{ (x,y);\,x\in\mathbb{R}^{n},y>0\right\} $
and its boundary $\partial\mathbb{R}_{+}^{n+1}:=\left\{ (x,0);\,x\in\mathbb{R}^{n}\right\} $.
Given any $x_{0}\in\mathbb{R}^{n}$, $(x_{0},y_{0})\in\mathbb{R}^{n+1}$
and $R>0$, we denote the balls 
\begin{align*}
 & B\left(x_{0},R\right):=\left\{ x\in\mathbb{R}^{n}:\mbox{ }|x-x_{0}|<R\right\} \subset\mathbb{R}^{n},\\
 & B^{n+1}\left((x_{0},y_{0}),R\right):=\left\{ (x,y)\in\mathbb{R}^{n+1}:\mbox{ }\sqrt{|x-x_{0}|^{2}+|y-y_{0}|^{2}}<R\right\} ,
\end{align*}
and as $y_{0}=0$, we set 
\begin{align*}
 & B^{n+1}\left(x_{0},R\right):=B^{n+1}\left((x_{0},0),R\right),\\
 & B_{+}^{n+1}\left(x_{0},R\right):=B^{n+1}\left(x_{0},R\right)\cap\{y>0\},\\
 & B^{*}\left(x_{0},R\right):=B^{n+1}\left(x_{0},R\right)\cap\{y=0\}.
\end{align*}

Let $s\in(0,1)$ and $\mathcal{D}$ be a bounded Lipschitz domain
in $\mathbb{R}^{n+1}$. Let $w$ be an arbitrary $A_{2}$ Muchenhoupt
weight function (cf. \cite{fabes1982local,muckenhoupt1972weighted})
and we denote $L^{2}(\mathcal{D},w)$ to be the weighted Sobolev space
containing all functions $U$ which are defined a.e. in $\mathcal{D}$
such that 
\[
\|U\|_{L^{2}(\mathcal{D},w)}:=\left(\int_{\mathcal{D}}w|U|^{2}dxdy\right)^{1/2}<\infty.
\]
Define 
\[
H^{1}(\mathcal{D},w):=\{U\in L^{2}(\mathcal{D},w);\,\nabla_{x,y}U\in L^{2}(\mathcal{D},w)\},
\]
where $\nabla_{x,y}:=(\nabla,\partial_{y})=(\nabla_{x},\partial_{y})$
is the total derivative in $\mathbb{R}^{n+1}$. In this work, the
weight function $w$ might be $y^{1-2s}$, $|y|^{1-2s}$, $y^{2s-1}$
and $|y|^{2s-1}$ and it is known (cf. \cite{kufner1987some}) that
$y^{1-2s}\in A_{2}$ for $s\in(0,1)$. It is easy to see that $L^{2}(\mathcal{D},w)$
and $H^{1}(\mathcal{D},w)$ are Banach spaces with respect to the
norms $\|\cdot\|_{L^{2}(\mathcal{D},w)}$ and 
\[
\|U\|_{H^{1}(\mathcal{D},w)}:=\left(\|U\|_{L^{2}(\mathcal{D},w)}^{2}+\|\nabla_{x,y}U\|_{L^{2}(\mathcal{D},w)}^{2}\right)^{1/2},
\]
respectively. We shall also make use of the weighted Sobolev space
$H_{0}^{1}(\mathcal{D},w)$ which is the closure of $C_{0}^{\infty}(\mathcal{D})$
under the $H^{1}(\mathcal{D},w)$ norm.

Let us consider the following extension problem in $\mathbb{R}^{n+1}$
\begin{equation}
\begin{cases}
-\mathcal{L}_{x}U+\frac{1-2s}{y}U_{y}+U_{yy}=0 & \mbox{ in }\mathbb{R}_{+}^{n+1},\\
U(\cdot,0)=u(x) & \mbox{ on }\partial\mathbb{R}_{+}^{n+1}.
\end{cases}\label{eq:Main Equ}
\end{equation}
The extension problem is related to the nonlocal equation \eqref{eq:nonlocalproblem},
where the nonlocal operator $\mathcal{L}^{s}$ has been regarded as
a Dirichlet-to-Neumann map of the above degenerate local problem \eqref{eq:Main Equ}.
For convenience, we introduce an auxiliary matrix-valued function
$\widetilde{A}:\mathbb{R}^{n}\to\mathbb{R}^{(n+1)\times(n+1)}$ by
\begin{equation}
\widetilde{A}(x)=\left(\begin{array}{cc}
A(x) & 0\\
0 & 1
\end{array}\right).\label{eq:An+1}
\end{equation}
We introduce the following degenerate local operator by 
\begin{equation}
\mathscr{L}_{\widetilde{A}}^{1-2s}=\nabla_{x,y}\cdot(y^{1-2s}\widetilde{A}(x)\nabla_{x,y}).\label{degenerate operator}
\end{equation}
It can be seen that $y^{-1+2s}\mathscr{L}_{\widetilde{A}}^{1-2s}$
is nothing but the above degenerate local operator introduced in (\ref{eq:Main Equ})
as

\[
\mathscr{L}_{\widetilde{A}}^{1-2s}=y^{1-2s}\left\{ \nabla\cdot(A(x)\nabla)+\frac{1-2s}{y}\partial_{y}+\partial_{y}^{2}\right\} .
\]

\subsection{Basic properties for the extension problem}

Let us begin with the following solvability result of the extension
for $\mathcal{L}$, where $\mathcal{L}$ is a second order elliptic
operator $\mathcal{L}=-\nabla\cdot(A(x)\nabla)$. Recall that the
fractional Sobolev space $H^{s}(\mathbb{R}^{n})$ can be realized
as a trace space of the weighted Sobolev space $H^{1}(\mathbb{R}_{+}^{n+1},y^{1-2s})$
for $s\in(0,1)$ (see \cite{tyulenev2014description}), i.e., for
a given $u\in H^{s}(\mathbb{R}^{n})$, there exists $U_{0}(x,y)\in H^{1}(\mathbb{R}_{+}^{n+1},y^{1-2s})$
such that $U_{0}(x,0)=u(x)\in H^{s}(\mathbb{R}^{n})$ with 
\begin{equation}
\|U_{0}\|_{H^{1}(\mathbb{R}_{+}^{n+1},y^{1-2s})}\leq C\|u\|_{H^{s}(\mathbb{R}^{n})}.\label{Trace estimate}
\end{equation}
For given $u\in H^{s}(\mathbb{R}^{n})$ and define $H_{0}^{1}(\mathbb{R}_{+}^{n+1},y^{1-2s}):=\{U\in H^{1}(\mathbb{R}_{+}^{n+1},y^{1-2s}):\quad U=0\mbox{ on }\partial\mathbb{R}_{+}^{n+1}\}$,
then we say $U(x,y)\in H^{1}(\mathbb{R}_{+}^{n+1},y^{1-2s})$ is a
weak solution of the Dirichlet boundary value problem (\ref{eq:Main Equ})
whenever $\mathscr{U}:=U-U_{0}\in H_{0}^{1}(\mathbb{R}_{+}^{n+1},y^{1-2s})$
\begin{align}
 & \int_{\mathbb{R}_{+}^{n+1}}y{}^{1-2s}\widetilde{A}(x)\nabla_{x,y}\mathscr{U}\cdot\nabla_{x,y}\phi dxdy\label{eq:some equation}\\
= & -\int_{\mathbb{R}_{+}^{n+1}}y{}^{1-2s}\widetilde{A}(x)\nabla_{x,y}U_{0}\cdot\nabla_{x,y}\phi dxdy,\nonumber 
\end{align}
 for all $\phi\in C_{c}^{\infty}(\mathbb{R}_{+}^{n+1})$. The solution
$U(x,y)\in H^{1}(\mathbb{R}_{+}^{n+1},y^{1-2s})$ can be also characterized
as a unique minimizer of the Dirichlet functional 
\[
\min_{\Psi\in H^{1}(\mathbb{R}_{+}^{n+1},y^{1-2s})}\left\{ \int_{\mathbb{R}_{+}^{n+1}}y^{1-2s}\widetilde{A}(x)\nabla_{x,y}\Psi\cdot\nabla_{x,y}\Psi dxdy:\quad\Psi(x,0)=u(x)\right\} .
\]
The existence and the uniqueness for the Dirichlet problem with zero
exterior data is given in the Appendix. First, we have the following
uniqueness result.
\begin{lem}
\label{Lem unique extension} (Unique extension) Let $s\in(0,1)$
and let $\widetilde{A}$ be given by \eqref{eq:An+1} with $A(x)$
satisfying \eqref{eq:ellipticity and symmetry condition}. Given any
$u\in H^{s}(\mathbb{R}^{n})$, there exists a unique solution $U\in H^{1}(\mathbb{R}_{+}^{n+1},y^{1-2s})$
of 
\begin{equation}
\begin{cases}
\mathscr{L}_{\widetilde{A}}^{1-2s}U=0 & \mbox{ in }\mathbb{R}_{+}^{n+1},\\
U(\cdot,0)=u & \mbox{ in }\mathbb{R}^{n}.
\end{cases}\label{eq:degenerateBVP}
\end{equation}
\end{lem}
\begin{proof}
It is known from \cite{nekvinda1993characterization,tyulenev2014description}
that $H^{s}(\mathbb{R}^{n})$ can be regarded as the trace space of
$H^{1}(\mathbb{R}_{+}^{n+1},y^{1-2s})$ on $\partial\mathbb{R}_{+}^{n+1}$.
Therefore, one can find a function $V$ in the space $H^{1}(\mathbb{R}_{+}^{n+1},y^{1-2s})$
such that $V(x,0)=u(x)$ for $x\in\mathbb{R}^{n}$. It is then verified
by \cite[Theorem 2.2]{fabes1982local} that there is a unique solution
$U\in H^{1}(\mathbb{R}_{+}^{n+1},y^{1-2s})$ of $\mathscr{L}_{\widetilde{A}}^{1-2s}U=0$
such that $U-V\in H_{0}^{1}(\mathbb{R}_{+}^{n+1},y^{1-2s})$. Thus,
the existence of solution to \eqref{eq:degenerateBVP} has been proven.
The uniqueness is then a simple consequence of \cite[Theorem 2.2]{fabes1982local}. 
\end{proof}
Next, we demonstrate the following stability estimate.
\begin{lem}
(Stability estimate) Let $u$ and $U$ be the same as in Lemma \eqref{Lem unique extension},
then the stability estimate is given by 
\begin{equation}
\|U\|_{H^{1}(\mathbb{R}_{+}^{n+1},y^{1-2s})}\leq C\|u\|_{H^{s}(\mathbb{R}^{n})}.\label{Stability estimate}
\end{equation}
for some $C>0$ independent of $u$ and $U$.\end{lem}
\begin{proof}
Given $u\in H^{s}(\mathbb{R}^{n})$, there exists $U_{0}(x,y)\in H^{1}(\mathbb{R}_{+}^{n+1},y^{1-2s})$
such that $U_{0}(x,0)=u(x)$. Since $U\in H^{1}(\mathbb{R}_{+}^{n+1},y^{1-2s})$
is a weak solution of (\ref{eq:degenerateBVP}), let $V:=U-U_{0}$,
then $V\in H^{1}(\mathbb{R}_{+}^{n+1},y^{1-2s})$ is a weak solution
of 
\[
\begin{cases}
\nabla\cdot(y^{1-2s}\widetilde{A}\nabla V)=\nabla\cdot G & \mbox{ in }\mathbb{R}_{+}^{n+1},\\
V(x,0)=0 & \mbox{ in }\mathbb{R}^{n},
\end{cases}
\]
where $G:=-y{}^{1-2s}\widetilde{A}(x)\nabla_{x,y}U_{0}$. It is easy
to see that $y^{2s-1}G\in L^{2}(\mathbb{R}_{+}^{n+1},y^{1-2s})$ and
\begin{eqnarray*}
\int_{\mathbb{R}_{+}^{n+1}}y^{1-2s}|y^{2s-1}G|^{2}dxdy & = & \int_{\mathbb{R}_{+}^{n+1}}y^{1-2s}\left|\widetilde{A}\nabla_{x,y}U_{0}\right|^{2}dxdy\\
 & \leq & C\int_{\mathbb{R}_{+}^{n+1}}y^{1-2s}|\nabla_{x,y}U_{0}|^{2}dxdy,
\end{eqnarray*}
for some universal constant $C>0$. Thus, from \eqref{eq:stabilityA}
in Appendix, we know 
\begin{eqnarray*}
\|U\|_{H^{1}(\mathbb{R}_{+}^{n+1},y^{1-2s})} & \leq & C\|y^{-1+2s}G\|_{L^{2}(\mathbb{R}_{+}^{n+1},y^{1-2s})}\\
 & \leq & C\|\mathscr{U}_{0}\|_{H^{1}(\mathbb{R}_{+}^{n+1},y^{1-2s})}\\
 & \leq & C\|u\|_{H^{s}(\mathbb{R}^{n})},
\end{eqnarray*}
for some constant $C>0$ and the last inequality comes the trace estimate
(\ref{Trace estimate}).
\end{proof}
We observe that since $A(x)$ satisfies the hypothesis ($\mathcal{H}$),
from the standard elliptic regularity theory, we get that $U$ is
$C^{\infty}$-smooth in $\mathbb{R}_{+}^{n+1}$. Consequently, by
using the standard weak formulation method, we can obtain that $y^{1-2s}\partial_{y}U$
converges to some function $h\in H^{-s}(\mathbb{R}^{n})$ as $y\to0$
in $H^{-s}(\mathbb{R}^{n})$ as 
\begin{equation}
(h,\phi(x,0))_{H^{-s}(\mathbb{R}^{n})\times H^{s}(\mathbb{R}^{n})}=\int_{\mathbb{R}_{+}^{n+1}}y^{1-2s}\widetilde{A}(x)\nabla_{x,y}U\cdot\nabla_{x,y}\phi dxdy,\label{eq:Nman}
\end{equation}
for all $\phi\in H^{1}(\mathbb{R}_{+}^{n+1},y^{1-2s})$. In other
words, $U\in H^{1}(\mathbb{R}_{+}^{n+1},y^{1-2s})$ is a weak solution
of the Neumann boundary value problem 
\begin{equation}
\begin{cases}
\nabla_{x,y}\cdot(y^{1-2s}\widetilde{A}(x)\nabla_{x,y}U)=0 & \mbox{ in }\mathbb{R}_{+}^{n+1},\\
\lim_{y\to0}y^{1-2s}\partial_{y}U=h & \mbox{ in }\mathbb{R}^{n}\times\{0\}.
\end{cases}\label{Extension and DN map}
\end{equation}
The following proposition characterizes $\lim_{y\to0}y^{1-2s}\partial_{y}U=h$
as $d_{s}h=\mathcal{L}^{s}u$, for some constant $d_{s}$ depending
on $s$, which connects the nonlocal problem and the extension problem.
\begin{prop}
\label{Extension lemma in the weak sense} Given $u\in H^{s}(\mathbb{R}^{n})$,
define 
\begin{equation}
U(x,y):=\int_{\mathbb{R}^{n}}P_{y}^{s}(x,z)u(z)dz,\label{eq:Vxy}
\end{equation}
where $P_{y}^{s}$ is the Poisson kernel given by 
\begin{equation}
P_{y}^{s}(x,z)=\dfrac{y^{2s}}{4^{s}\Gamma(s)}\int_{0}^{\infty}e^{-\frac{y^{2}}{4t}}p_{t}(x,z)\dfrac{dt}{t^{1+s}},\quad x,z\in\mathbb{R}^{n},\,y>0,\label{eq:representation of Poisson kernel}
\end{equation}
with the heat kernel $p_{t}$ introduced in Section \ref{sec:LsHsHeatkernel}.
Then $U\in H^{1}(\mathbb{R}_{+}^{n+1},y^{1-2s})$ and is the weak
solution of \eqref{eq:degenerateBVP} and 
\begin{equation}
\lim_{y\to0^{+}}\dfrac{U(\cdot,y)-U(\cdot,0)}{y^{2s}}=\frac{1}{2s}\lim_{y\to0+}y^{1-2s}\partial_{y}U(\cdot,y)=\frac{\Gamma(-s)}{4^{s}\Gamma(s)}\mathcal{L}^{s}u,\label{Almost finishing limit}
\end{equation}
in $H^{-s}(\mathbb{R}^{n})$.\end{prop}
\begin{rem}
Note that Stinga and Torrea \cite{stinga2010extension} proved the
equality \eqref{Almost finishing limit} for $u\in\mathrm{Dom}(\mathcal{L}^{s})$.
Here we extend such results for all $u\in H^{s}(\mathbb{R}^{n})$. \end{rem}
\begin{proof}
From \cite[Theorem 2.1]{stinga2010extension}, we know that 
\[
U(x,y):=\int_{\mathbb{R}^{n}}P_{y}^{s}(x,z)u(z)dz
\]
solves the equation \eqref{eq:degenerateBVP}. From Lemma \ref{Lem unique extension},
we know that $U\in H^{1}(\mathbb{R}_{+}^{n+1},y^{1-2s})$ due to $u\in H^{s}(\mathbb{R}^{n})$.
It remains to demonstrate that \eqref{Almost finishing limit} for
$u\in H^{s}(\mathbb{R}^{n})$. 

Firstly, we prove $\lim_{y\to0^{+}}\dfrac{U(\cdot,y)-U(\cdot,0)}{y^{2s}}=\dfrac{\Gamma(-s)}{4^{s}\Gamma(s)}\mathcal{L}^{s}u$,
for any $v\in H^{s}(\mathbb{R}^{n})$. By using \eqref{eq:Py>0} in
Appendix, we can deduce 
\[
\begin{split} & \left\langle \lim_{y\to0}\dfrac{U(\cdot,y)-U(\cdot,0)}{y^{2s}},v\right\rangle _{H^{-s}(\mathbb{R}^{n})\times H^{s}(\mathbb{R}^{n})}\\
= & \int_{\mathbb{R}^{n}}\lim_{y\to0}\left[\frac{\int_{\mathbb{R}^{n}}P_{y}^{s}(x,z)u(z)dz-u(x)}{y^{2s}}\right]v(x)dx\\
= & \int_{\mathbb{R}^{n}}\lim_{y\to0}\frac{\int_{\mathbb{R}^{n}}P_{y}^{s}(x,z)\left(u(z)-u(x)\right)v(x)dz}{y^{2s}}dx\\
= & \dfrac{1}{4^{s}\Gamma(s)}\int_{\mathbb{R}^{n}}\lim_{y\to0^{+}}\lim_{\epsilon\to0^{+}}\int_{|z-x|>\epsilon}\int_{0}^{\infty}e^{-\frac{y^{2}}{4t}}p_{t}(x,z)\left(u(z)-u(x)\right)v(x)\dfrac{dt}{t^{1+s}}dzdx\\
= & \frac{1}{4^{s}\Gamma(s)}\int_{\mathbb{R}^{n}}\lim_{\epsilon\to0^{+}}\int_{|z-x|>\epsilon}\left(\int_{0}^{\infty}p_{t}(x,z)\dfrac{dt}{t^{1+s}}\right)\left(u(z)-u(x)\right)v(x)dzdx\\
= & \frac{\Gamma(-s)}{4^{s}\Gamma(s)}\int_{\mathbb{R}^{n}}\lim_{\epsilon\to0^{+}}\int_{|z-x|>\epsilon}\mathcal{K}_{s}(x,z)\left(u(z)-u(x)\right)v(x)dzdx\\
= & \frac{\Gamma(-s)}{4^{s}\Gamma(s)}\left\langle \mathcal{L}^{s}u,v\right\rangle _{H^{-s}(\mathbb{R}^{n})\times H^{s}(\mathbb{R}^{n})}.
\end{split}
\]

Secondly, we prove $\frac{1}{2s}\lim_{y\to0+}y^{1-2s}\partial_{y}U(\cdot,y)=\dfrac{\Gamma(-s)}{4^{s}\Gamma(s)}\mathcal{L}^{s}u$
by utilizing the density argument. It is known that for $u\in\mathrm{Dom}(\mathcal{L}^{s})$,
then \eqref{Almost finishing limit} holds in $L^{2}(\mathbb{R}^{n})$.
Consider a sequence $\{u_{k}\}_{k\in\mathbb{N}}\subseteq\mathrm{Dom}(\mathcal{L}^{s})\cap H^{s}((\mathbb{R}^{n})$
such that $u_{k}\to u$ in $H^{s}(\mathbb{R}^{n})$ as $k\to\infty$.
Let $U_{k}\in H^{1}(\mathbb{R}_{+}^{n+1},y^{1-2s})$ be the solution
to \eqref{eq:Main Equ} with the boundary data $u_{k}$ for each $k\in\mathbb{N}$.
Recall from \eqref{eq:DomL^s subset DomL} and \eqref{eq:DomL} that
$C_{0}^{\infty}(\mathbb{R}^{n})\subseteq\mathrm{Dom}(\mathcal{L}^{s})$.
Thus by \cite[Theorem 1.1]{stinga2010extension} and Lemma \ref{Lem unique extension},
$U_{k}(x,y)$ can be uniquely represented by 
\begin{equation}
U_{k}(x,y)=\int_{\mathbb{R}^{n}}P_{y}^{s}(x,z)u_{k}(z)dz,\mbox{ for }x\in\mathbb{R}^{n}\mbox{ and }y\in\mathbb{R}_{+}.\label{eq:Vkxy}
\end{equation}
Moreover, the following relation 
\begin{equation}
\frac{1}{2s}\lim_{y\to0+}y^{1-2s}\partial_{y}U_{k}(\cdot,y)=\frac{\Gamma(-s)}{4^{s}\Gamma(s)}\mathcal{L}^{s}u_{k}
\end{equation}
holds in $L^{2}(\mathbb{R}^{n})$ by \cite[Theorem 1.1]{stinga2010extension}
again. Following that, from \eqref{eq:Nman} or \eqref{Extension and DN map},
we conclude for being $U\in H^{1}(\mathbb{R}_{+}^{n+1},y^{1-2s})$,
$\frac{1}{2s}\lim_{y\to0+}y^{1-2s}\partial_{y}U(x,y)=h$ exists in
$H^{-s}(\mathbb{R}^{n})$. For convenience, we set $h_{k}:=\frac{\Gamma(-s)}{4^{s}\Gamma(s)}\mathcal{L}^{s}u_{k}$.
Note that $U_{k}-U$ solves the equation $\mathscr{L}_{\widetilde{A}}^{1-2s}(U_{k}-U)=0$
in $\mathbb{R}_{+}^{n+1}$ with $(U_{k}-U)(x,0)=(u_{k}-u)(x)$, by
the stability estimate (\ref{Stability estimate}), we get 
\[
\|U_{k}-U\|_{H^{1}(\mathbb{R}_{+}^{n+1},y^{1-2s})}\leq C\|u_{k}-u\|_{H^{s}(\mathbb{R}^{n})},
\]
for some constant $C>0$ independent of $U_{k}$, $U$, $u_{k}$ and
$u$. Hence, $U_{k}\to U$ in $H^{1}(\mathbb{R}_{+}^{n+1},y^{1-2s})$
due to $u_{k}\to u$ in $H^{s}(\mathbb{R}^{n})$ as $k\to\infty$.
On the other hand, by using the weak formulation \eqref{eq:Nman},
for any $\phi\in H^{1}(\mathbb{R}_{+}^{n+1},y^{1-2s})$ (recall that
$\phi(x,0)\in H^{s}(\mathbb{R}^{n})$ by the trace characterization),
$U_{k}-U$ satisfies 
\[
\int_{\mathbb{R}_{+}^{n+1}}y^{1-2s}\widetilde{A}(x)\nabla_{x,y}(U_{k}-U)\cdot\nabla_{x,y}\phi dxdy=\left((h_{k}-h),\phi(x,0)\right)_{H^{-s}(\mathbb{R}^{n})\times H^{s}(\mathbb{R}^{n})}.
\]
From $U_{k}\to U$ in $H^{1}(\mathbb{R}_{+}^{n+1},y^{1-2s})$ as $k\to\infty$,
we conclude that $h_{k}\to h$ in $H^{-s}(\mathbb{R}^{n})$ as $k\to\infty$.
Finally, by the integral representation (\ref{eq:integral represent for nonlocal})
for $\mathcal{L}^{s}$ and $u_{k}\to u$ in $H^{s}(\mathbb{R}^{n})$
as $k\to\infty$, we can derive that 
\begin{eqnarray*}
 &  & \left(h_{k},\phi(x,0)\right)_{H^{-s}(\mathbb{R}^{n})\times H^{s}(\mathbb{R}^{n})}\\
 & = & \frac{\Gamma(-s)}{4^{s}\Gamma(s)}\left(\mathcal{L}^{s}u_{k},\phi(x,0)\right)_{H^{-s}(\mathbb{R}^{n})\times H^{s}(\mathbb{R}^{n})}\\
 & = & \frac{1}{2}\frac{\Gamma(-s)}{4^{s}\Gamma(s)}\int_{\mathbb{R}^{n}\times\mathbb{R}^{n}}(u_{k}(x)-u_{k}(z))(\phi(x,0)-\phi(z,0))\mathcal{K}_{s}(x,z)dxdz\\
 & \to & \frac{1}{2}\frac{\Gamma(-s)}{4^{s}\Gamma(s)}\int_{\mathbb{R}^{n}\times\mathbb{R}^{n}}(u(x)-u(z))(\phi(x,0)-\phi(z,0))\mathcal{K}_{s}(x,z)dxdz\\
 & = & \frac{\Gamma(-s)}{4^{s}\Gamma(s)}\left(\mathcal{L}^{s}u,\phi(x,0)\right)_{H^{-s}(\mathbb{R}^{n})\times H^{s}(\mathbb{R}^{n})},
\end{eqnarray*}
as $k\to\infty$. By using the uniqueness limit of $h_{k}$, then
we can conclude that $h=\frac{\Gamma(-s)}{4^{s}\Gamma(s)}\mathcal{L}^{s}u$
in $H^{-s}(\mathbb{R}^{n})$. Therefore we have verified \eqref{Almost finishing limit}
and thus completes the proof. 
\end{proof}
Next, we recall the well-known reflection extension for the extension
problem.

\subsection{Even reflection extension and its related regularity properties}

Similar to the fractional Laplacian case, we also have the following
reflection property for fractional variable operators. Let us consider
$B_{+}^{n+1}(x_{0},R)\subset\mathbb{R}_{+}^{n+1}$ , where $\lim_{y\to0}y^{1-2s}\partial_{y}U=0$
in $B(x_{0},R)\subset\Omega$. Now by using the even reflection, we
define 
\begin{equation}
\widetilde{U}(x,y):=\begin{cases}
U(x,y) & \mbox{ if }y\geq0,\\
U(x,-y) & \mbox{ if }y<0,
\end{cases}\label{eq:even reflection}
\end{equation}
where $U(x,y)$ solves (\ref{eq:Main Equ}). Then $\widetilde{U}(x,y)$
is a solution of the following problem 
\begin{equation}
\nabla_{x,y}\cdot(|y|^{1-2s}\widetilde{A}(x)\nabla_{x,y}\widetilde{U})=0\mbox{ in }B^{n+1}(x_{0},R).\label{eq:reflection extension}
\end{equation}
In general, for 
\[
\nabla_{x,y}\cdot(|y|^{1-2s}\widetilde{A}(x)\nabla_{x,y}V)=\nabla_{x,y}\cdot G\mbox{ in }B^{n+1}(x_{0},R),
\]
where $G$ is a vector-valued function satisfies $y^{-1+2s}|G|\in L^{2}(B^{n+1}(x_{0},R),|y|^{1-2s})$,
then we say $V$ to be a weak solution of the above equation if 
\[
\int_{B^{n+1}(x_{0},R)}|y|^{1-2s}\widetilde{A}(x)\nabla_{x,y}V\cdot\nabla_{x,y}\phi dxdy=\int_{B^{n+1}(x_{0},R)}G\cdot\nabla_{x,y}\phi dxdy,
\]
for all $\phi\in C_{c}^{\infty}(B^{n+1}(x_{0},R))$.

We have the following regularity results.
\begin{prop}
\label{Prop Holder regularity for degenerate}(a) (Solvability) Let
$D\subset\mathbb{R}^{n+1}$ be a bounded domain with $C^{\infty}$-smooth
boundary and $G$ is a vector-valued function satisfies $|y|^{-1+2s}|G|\in L^{2}(D,|y|^{1-2s})$.
Let $g\in H^{1}(D,|y|^{1-2s})$, then there is a unique solution $V\in H^{1}(D,|y|^{1-2s})$
of 
\begin{equation}
\nabla_{x,y}\cdot(|y|^{1-2s}\widetilde{A}(x)\nabla_{x,y}V)=\nabla_{x,y}\cdot G\mbox{ in }D,\label{eq:degenerate elliptic problem}
\end{equation}
with $V-g\in H_{0}^{1}(D,|y|^{1-2s})$. 

(b) (Interior Hölder's regularity) Let $D\subset\mathbb{R}^{n+1}$
be a bounded domain with $C^{\infty}$-smooth boundary and let $V$
be a weak solution to (\ref{eq:degenerate elliptic problem}) where
$h$ is a vector-valued function satisfies $y^{-1+2s}|G|\in L^{2(n+1)}(D,|y|^{1-2s})$.
Then $V\in C^{0,\beta}(D')$ for some $\beta\in(0,1)$ depending on
$n$ and $s$, where $D'\Subset D$ is an arbitrary open set.

(c) (Higher regularity in the $x$-direction) Let $D\subset\mathbb{R}^{n+1}$
be a bounded domain with $C^{\infty}$-smooth boundary and let $V\in H^{1}(D,|y|^{1-2s})$
be a weak solution of 
\begin{equation}
\nabla_{x,y}\cdot(|y|^{1-2s}\widetilde{A}(x)\nabla_{x,y}V)=0\mbox{ in }D.\label{eq:0-source equation}
\end{equation}
Then for each fixed $y=y_{0}$ with $(x,y_{0})\in D'$, we have $V(x,y_{0})\in C^{\infty}(D'\cap\{y=y_{0}\})$
in any open subset $D'\Subset D$, whenever the matrix $A(x)$ satisfies
the hypothesis ($\mathcal{H}$). \end{prop}
\begin{proof}
The proof of (a) has been established in \cite{fabes1982local} for
$|y|^{1-2s}$ being an $A_{2}$ function and the proof of (b) is a
direct consequence of \cite[Theorem 2.3.12]{fabes1982local}. We move
into showing (c). Set $\Delta_{x_{i}}^{h}$ to be the classical differential
quotient operator, which means 
\[
\Delta_{x_{i}}^{h}V(x,y):=\dfrac{V(x+he_{i},y)-V(x,y)}{h}\mbox{ for any }i=1,2,\cdots,n,
\]
where we have fixed $y>0$. From straightforward calculation, if $V(x)$
solves \eqref{eq:0-source equation}, we get $\Delta_{x_{i}}^{h}V(x)$
solves 
\begin{equation}
\nabla_{x,y}\cdot(|y|^{1-2s}\widetilde{A}(x+he_{i})\nabla_{x,y}(\Delta_{x_{i}}^{h}V))=\nabla_{x,y}\cdot G\mbox{ in }D_{h},\label{eq:iteration 1}
\end{equation}
where the function $G=-|y|^{1-2s}(\Delta_{x_{i}}^{h}\widetilde{A})\nabla_{x,y}V(x,y)$
and $D_{h}\Subset D$ is an arbitrary subset such that the Hausdorff
distance between $D$ and $D_{h}$ greater than $h$. Note that in
the right hand side of (\ref{eq:iteration 1}) satisfies the condition
$|y|^{-1+2s}G\in L^{2}(D_{h},|y|^{1-2s})$ since $V\in H^{1}(D,|y|^{1-2s})$.
Hence, from (a) and the standard cutoff techniques, we know that $\Delta_{x_{i}}^{h}V\in H^{1}(D_{h},|y|^{1-2s})$
and 
\begin{equation}
\|\Delta_{x_{i}}^{h}V\|_{H^{1}(D',|y|^{1-2s})}\leq K<\infty,\label{Differential quotient 1}
\end{equation}
for any subset $D'\Subset D_{h}$ and it is easy to see that the constant
$K>0$ independent of $h$ since $A(x)$ satisfies the hypothesis
($\mathcal{H}$) such that $\left|\Delta_{x_{i}}^{h}\widetilde{A}(x)\right|\leq\|\partial_{x_{i}}\widetilde{A}\|_{L^{\infty}(D)}\leq C<\infty$.
Recall that $H^{1}(D,|y|^{1-2s})$ is a reflexive Banach space, by
using the same argument as in \cite[Lemma 7.24]{gilbarg2015elliptic},
then $\partial_{x_{i}}V\in H^{1}(D',|y|^{1-2s})$ for $i=1,2,\cdots,n$
with $\|\nabla_{x,y}(\partial_{x_{i}}V)\|_{L^{2}(D',|y|^{1-2s})}\leq K<\infty$,
where $K>0$ is the same constant as in \eqref{Differential quotient 1}.
Continue this process, we can apply the differential quotient with
respect to the $x$-direction for any order, then one can derive that
$V\in H_{x}^{\infty}(D',|y|^{1-2s})$, where 
\begin{equation}
H_{x}^{\infty}(D',|y|^{1-2s}):=\left\{ U\in H^{1}(D',|y|^{1-2s}):\ \partial_{x}^{\alpha}U\in H^{1}(D',|y|^{1-2s})\right\} ,\label{x-direction weighted space}
\end{equation}
for any multi-index $\alpha\in(\mathbb{N}\cup\{0\})^{n}$. Now, since
$\partial_{x}^{\alpha}V(x,y)\in H^{1}(D',|y|^{1-2s})$, by using the
trace theorem for the weighted Sobolev space again, we have $\partial_{x}^{\alpha}V(x,0)\in H^{s}(D'\cap\{y=0\})$
for any $\alpha\in\mathbb{N}^{n}$, or $V(x,0)\in H^{m+s}(D'\cap\{y=0\})$
for any $m\in\mathbb{N}\cup\{0\}$. Now, apply the fractional Sobolev
embedding theorem (see \cite{di2012hitchhiks} for instance), we derive
$V(x,0)\in C^{\infty}(D'\cap\{y=0\})$. For each fixed $y=y_{0}\neq0$,
the equation \eqref{eq:0-source equation} can be regarded as a standard
second order elliptic equation with $C^{\infty}$-smooth coefficients,
by the standard elliptic theory it is easy to see  that $V(x,y_{0})$
is $C^{\infty}$-smooth in $D'\cap\{y=y_{0}\}$ with respect to $x$.
This finishes the proof.
\end{proof}
In the end of this section, we introduce the conjugate equation, which
is associated to the degenerate operator $\mathscr{L}_{\widetilde{A}}^{1-2s}$
given by (\ref{degenerate operator}).

\subsection{Conjugate equation and odd reflection}

As in \cite[Section 2]{caffarelli2007extension} and \cite[Section 2]{stinga2010extension},
it is known that if $U\in H^{1}(\mathbb{R}_{+}^{n+1},y^{1-2s})$ is
a weak solution to $\mathscr{L}_{\widetilde{A}}^{1-2s}U=0$ in $\mathbb{R}_{+}^{n+1}$
then the function 
\[
W(x,y):=y^{1-2s}\partial_{y}U(x,y)
\]
is a solution to the conjugate equation 
\begin{equation}
\mathscr{L}_{\widetilde{A}}^{-1+2s}W=-\mathcal{L}_{x}W+\frac{1-2s}{y}W_{y}+W_{yy}=0,\mbox{ in }\mathbb{R}_{+}^{n+1}.\label{eq:conjugate equation}
\end{equation}
If we assume that $W(x,0)=0$ for $x\in B(x_{0},R)\subset\Omega$
and use the odd reflection, we define 
\begin{equation}
\widetilde{W}(x,y):=\begin{cases}
W(x,y), & \mbox{ if }y\geq0,\\
-W(x,-y), & \mbox{ if }y<0.
\end{cases}\label{eq:odd reflection}
\end{equation}
Then we will prove that that $\widetilde{W}\in H^{1}(|y|^{-1+2s},B^{n+1}(x_{0},R))$
is a weak solution of 
\[
\nabla_{x,y}\cdot(|y|^{-1+2s}\widetilde{A}(x)\nabla_{x,y}\widetilde{W})=0\mbox{ in }B^{n+1}(x_{0},R).
\]
By using Proposition \ref{Prop Holder regularity for degenerate},
we say $\widetilde{W}\in C^{0,\beta}(B^{n+1}(x_{0},R))$ for some
$\beta\in(0,1)$ depending on $n$ and $1-2s$.
\begin{lem}
(The Conjugate Equation)\label{lem:The Conjugate Equation} Let $U\in H^{1}(\mathbb{R}_{+}^{n+1},y^{1-2s})$
be a weak solution of $\mathscr{L}_{\widetilde{A}}^{1-2s}U=0$ in
$\mathbb{R}_{+}^{n+1}$ and $\lim_{y\to0}y^{1-2s}\partial_{y}U=0$
in $B^{*}(x_{0},R)$. Then for any $r<R$, the function $W=y^{1-2s}\partial_{y}U\in H^{1}(B_{+}^{n+1}(x_{0},r),y^{2s-1})$
solves the conjugate equation $\mathscr{L}_{\widetilde{A}}^{2s-1}W=0$
weakly in $B_{+}^{n+1}(x_{0},r)$ with $W(x,0)=0$ for $x\in B(x_{0},r)$.\end{lem}
\begin{proof}
As previous discussions, we define $\widetilde{U}$ and $\widetilde{W}$
to be even and odd extension by (\ref{eq:even reflection}) and (\ref{eq:odd reflection}),
respectively. Then $\widetilde{U}$ solves 
\[
\nabla_{x,y}\cdot(|y|^{1-2s}\widetilde{A}(x)\nabla_{x,y}\widetilde{U})=0\mbox{ in }B^{n+1}(x_{0},R),
\]
where we know the fact that $\widetilde{U}\in C^{0,\beta}(B^{n+1}(x_{0},R))$
and $\widetilde{U}(\cdot,y)\in C^{\infty}(B(x_{0},R))$ for $y>0$.
We will show that $\widetilde{W}\in H_{loc}^{1}(B^{n+1}(x_{0},R),|y|^{2s-1})$
as a weak solution of 
\begin{equation}
\nabla_{x,y}\cdot(|y|^{2s-1}\widetilde{A}(x)\nabla_{x,y}\widetilde{W})=0\mbox{ in \ensuremath{\mathcal{B}\Subset}}B^{n+1}(x_{0},R),\label{eq:W}
\end{equation}

First, it is easy to see that $\widetilde{W}\in L^{2}(B^{n+1}(x_{0},R),|y|^{2s-1})$.
Second, for $0<h\ll1$, we consider the differential quotient for
$\widetilde{U}$, then $\Delta_{x_{i}}^{h}\widetilde{U}\in H^{1}(B^{n+1}(x_{0},R),|y|^{1-2s})$
is a weak solution of 
\begin{equation}
\nabla_{x,y}\cdot(|y|^{1-2s}\widetilde{A}(x)\nabla_{x,y}(\Delta_{x_{i}}^{h}\widetilde{U}))=\nabla_{x,y}\cdot G\mbox{ in }B^{n+1}(x_{0},R),\label{one differentiation equation}
\end{equation}
for any $h>0$, where the function $G=-|y|^{1-2s}(\Delta_{x_{i}}^{h}\widetilde{A})\nabla_{x,y}\widetilde{U}$
and it is easy to see that $|y|^{2s-1}H\in L^{2}(B^{n+1}(x_{0},R),|y|^{1-2s}))$
for $i=1,2,\cdots,n$. Let $\eta\in C_{c}^{\infty}(B^{n+1}(x_{0},R))$
be a standard cutoff function such that $0\leq\eta\leq1$ and 
\[
\eta(x,y)=\begin{cases}
1 & \mbox{ for }x\in B^{n+1}(x_{0},\frac{3}{4}R),\\
0 & \mbox{ for }x\notin B^{n+1}(x_{0},R),
\end{cases},\quad\|\nabla_{x,y}\eta\|_{L^{\infty}(B^{n+1}(x_{0},R))}\leq\dfrac{C}{R},
\]
for some constant $C>0$. Now, consider $\eta^{2}\Delta_{x_{i}}^{h}\widetilde{U}\in H^{1}(B^{n+1}(x_{0},R),|y|^{1-2s})$
as a test function and multiply it on the both sides of (\ref{one differentiation equation})
and do the integration by parts over $B^{n+1}(x_{0},R)$, then we
have 
\begin{align}
 & \int_{B^{n+1}(x_{0},R)}|y|^{1-2s}\widetilde{A}(x)\nabla_{x,y}(\Delta_{x_{i}}^{h}\widetilde{U})\cdot\nabla_{x,y}(\eta^{2}\Delta_{x_{i}}^{h}\widetilde{U})dxdy\nonumber \\
= & -\int_{B^{n+1}(x_{0},R)}|y|^{1-2s}(\Delta_{x_{i}}^{h}\widetilde{A})\nabla_{x,y}\widetilde{U}\cdot\nabla_{x,y}(\eta^{2}\Delta_{x_{i}}^{h}\widetilde{U}).\label{Caccipolli1}
\end{align}
From a direct computation, it is not hard to see that
\[
\nabla_{x,y}(\eta^{2}\Delta_{x_{i}}^{h}\widetilde{U})=\eta^{2}\nabla_{x,y}(\Delta_{x_{i}}^{h}\widetilde{U})+2(\eta\Delta_{x_{i}}^{h}\widetilde{U})\nabla_{x,y}\eta
\]
and (\ref{Caccipolli1}) becomes 
\begin{align}
 & \int_{B^{n+1}(x_{0},R)}|y|^{1-2s}\eta^{2}\left|\nabla_{x,y}(\Delta_{x_{i}}^{h}\widetilde{U})\right|^{2}dxdy\nonumber \\
\leq & C\int_{B^{n+1}(x_{0},R)}|y|^{1-2s}\left|\nabla_{x,y}(\Delta_{x_{i}}^{h}\widetilde{U})\right|\left|\eta\right|\left|\Delta_{x_{i}}^{h}\widetilde{U}\right|\left|\nabla_{x,y}\eta\right|dxdy\nonumber \\
 & +C\int_{B^{n+1}(x_{0},R)}|y|^{1-2s}|\eta|^{2}\left|\nabla_{x,y}\widetilde{U}\right|\left|\nabla_{x,y}(\Delta_{x_{i}}^{h}\widetilde{U})\right|dxdy\nonumber \\
 & +C\int_{B^{n+1}(x_{0},R)}|y|^{1-2s}\left|\nabla_{x,y}\widetilde{U}\right||\eta|\left|\nabla_{x,y}\eta\right|\left|\Delta_{x_{i}}^{h}\widetilde{U}\right|dxdy.\label{Caccipolli 2}
\end{align}
Apply the Young's inequality on \eqref{Caccipolli 2} and absorb the
highest order term of $\widetilde{U}$ to the left hand side of \eqref{Caccipolli 2},
then we can derive 
\begin{align*}
 & \int_{B^{n+1}(x_{0},\frac{3}{4}R)}|y|^{1-2s}\left|\nabla_{x,y}(\Delta_{x_{i}}^{h}\widetilde{U})\right|^{2}dxdy\\
\leq & C\left\{ \int_{B^{n+1}(x_{0},R)}|y|^{1-2s}\left|\Delta_{x_{i}}^{h}\widetilde{U}\right|^{2}dxdy+\int_{B^{n+1}(x_{0},R)}|y|^{1-2s}\left|\nabla_{x,y}\widetilde{U}\right|^{2}dxdy\right\} \\
\leq & C\|\widetilde{U}\|_{H^{1}(B^{n+1}(x_{0},R),|y|^{1-2s})},
\end{align*}
where the constant $C>0$ is independent of $\widetilde{U}$ and $h$.
This implies that $\Delta_{x_{i}}^{h}(\nabla_{x,y}\widetilde{U})\in L^{2}(B^{n+1}(x_{0},\dfrac{3}{4}R),|y|^{1-2s})$
and 
\[
\|\Delta_{x_{i}}^{h}(\nabla_{x,y}\widetilde{U})\|_{L^{2}(B^{n+1}(x_{0},\frac{3}{4}R),|y|^{1-2s})}\leq C,\mbox{ for }i=1,2,\cdots.n,
\]
for some constant $C>0$ is independent of $\widetilde{U}$ and $h$.
Then use the same technique as in Proposition \eqref{Prop Holder regularity for degenerate},
then one can conclude $\|\partial_{x_{i}}(\nabla_{x,y}\widetilde{U})\|_{L^{2}(B^{n+1}(x_{0},\frac{3}{4}R),|y|^{1-2s})}$,
which means $\partial_{x_{i}}\widetilde{U}\in H^{1}(B^{n+1}(x_{0},\dfrac{3}{4}R),|y|^{1-2s})$
for $i=1,2,\cdots,n$.

It remains to show $\partial_{y}W\in L^{2}(B_{+}^{n+1}(x_{0},\dfrac{3}{4}R),y{}^{2s-1})$.
Note that 
\[
\partial_{y}W=\partial_{y}(y^{1-2s}\partial_{y}U)=y^{1-2s}\left(\dfrac{1-2s}{y}\partial_{y}U+\partial_{y}^{2}U\right)=y^{1-2s}\mathcal{L}_{x}U.
\]
Via the fact that $\partial_{x_{i}x_{j}}^{2}\widetilde{U}\in L^{2}(B^{n+1}(x_{0},\dfrac{3}{4}R),|y|^{1-2s})$,
this implies the lemma holds and completes the proof. 
\end{proof}

\section{Strong unique continuation principle and the Runge approximation
property\label{Section 5}}

In order to prove Theorem \ref{thm:(Approximation-theorem)}, we will
be using the strong unique continuation principle (SUCP) for the extension
operator $\mathscr{L}_{\widetilde{A}}^{1-2s}$. Our strategy in proving
Theorem \ref{thm:(Approximation-theorem)} is decomposed into two
parts. First, we will prove under the condition of Theorem \ref{thm:(Approximation-theorem)},
the solution of the extension problem will vanish to infinite order,
which is inspired by the proof of \cite[Proposition 2.2]{ruland2015unique}.
Second, we apply the SUCP for degenerate differential equation, which
was introduced by \cite[Corollary 3.9]{yu2016unique}. Combine these
two steps, then we can prove the SUCP for the operator $\mathscr{L}_{\widetilde{A}}^{1-2s}$.

\subsection{Strong unique continuation principle}

We begin with the definition of the vanishing to infinity order for
the degenerate case. 
\begin{defn}
 (Vanishing to infinite order) A function $\Psi\in L_{loc}^{2}(\mathbb{R}_{+}^{n+1},y^{1-2s})$
is vanishing to infinite order at a point $(x_{0},0)\in\mathbb{R}_{+}^{n+1}$
if for every $m\in\mathbb{N}$, we have 
\begin{equation}
\lim_{r\to0}r^{-m}\int_{B^{n+1}(x_{0},r)}|y|^{1-2s}\Psi^{2}(x,y)\,dxdy=0.\label{eq:vanishing}
\end{equation}

\end{defn}
We begin with the first step: Vanishing to infinite order.
\begin{thm}
\label{Theorem vanishing to infinite order} Given $u\in H^{s}(\mathbb{R}^{n})$,
let $U\in H^{1}(\mathbb{R}_{+}^{n+1},y^{1-2s})$ be the unique solution
of the extension problem \eqref{eq:Main Equ}. Suppose that $u=\mathcal{L}^{s}u=0$
in $B(x_{0},2R)$. Then $U$ vanishes to infinite order on $B^{*}(x_{0},R)$.\end{thm}
\begin{proof}
We will follow ideas of proof of \cite[Proposition 2.2]{ruland2015unique}.

1. We know from Proposition \ref{Extension lemma in the weak sense}
that $U=\lim_{y\to0}y^{1-2s}\partial_{y}U=0$ in $B^{*}(x_{0},2R)$.
Define $W:=y^{1-2s}\partial_{y}U$. Then by Lemma~\ref{lem:The Conjugate Equation}
we know that $W\in H^{1}(B_{+}^{n+1}(x_{0},\dfrac{3}{2}R),y^{1-2s})$
solves $\mathscr{L}_{\widetilde{A}}^{2s-1}W=0$ in $B_{+}^{n+1}(x_{0},\dfrac{3}{2}R)$
with $W(x,0)=0$. We define $\widetilde{U}$ and $\widetilde{W}$
given by even reflection (\ref{eq:even reflection}) and odd reflection
(\ref{eq:odd reflection}), respectively. It is straightforwardly
verified that $\widetilde{U}\in H^{1}\left(B^{n+1}(x_{0},2R),|y|^{1-2s}\right)$
satisfies 
\begin{equation}
\nabla_{x,y}\cdot(|y|^{1-2s}\widetilde{A}(x)\nabla_{x,y}\widetilde{U})=0\mbox{ in }B^{n+1}(x_{0},2R),\label{eq:U equation in SUCP}
\end{equation}
and $\widetilde{W}\in H^{1}\left(B^{n+1}(x_{0},\dfrac{3}{2}R),|y|^{-1+2s}\right)$
is a solution to 
\begin{equation}
\nabla_{x,y}\cdot(|y|^{-1+2s}\widetilde{A}(x)\nabla_{x,y}\widetilde{W})=0\mbox{ in }B^{n+1}(x_{0},\dfrac{3}{2}R).\label{W equation in SUCP}
\end{equation}
Hence recalling by Proposition \ref{Prop Holder regularity for degenerate},
the functions $\widetilde{U}$ and $\widetilde{W}$ are Hölder continuous
in $B^{n+1}(x_{0},R)$. As a consequence $U$ and $W$ are both Hölder
continuous in $\overline{B_{+}^{n+1}(x_{0},R)}$. 

2. It can be seen by using the mean value theorem and the fundamental
theorem of calculus, for all $h\in C^{1}((0,1))\cap C([0,1])$ and
any $a\in(-\infty,1)$ if $h(0)=0$ and $\lim_{y\to0}y^{a}\dfrac{d}{dy}h(y)=0$
then 
\begin{equation}
\lim_{y\to0}y^{a-1}h(y)=0.\label{eq:Key ingredient of iteration}
\end{equation}
The remaining proof of this theorem follows the proof of \cite[Proposition 2.2]{ruland2015unique}.
We divided it into the following three steps arguments.

\textbf{}\\
\textbf{Step 1. One-step improvement}

As a solution to \eqref{eq:U equation in SUCP}, we know from Lemma~\ref{Prop Holder regularity for degenerate}
that $\widetilde{U}$ is $C^{0,\beta}$ in any direction in $\mathbb{R}^{n+1}$
and $C^{\infty}$ in the $x$-direction since $\widetilde{A}(x)$
is $C^{\infty}$-smooth. Thus, we can differentiate \eqref{eq:U equation in SUCP}
with respect to all $x_{i}$-direction up to an arbitrary order for
$i=1,2,\cdots,n$, due to the $C^{\infty}$-smoothness. By using the
continuity of $U$, we know that
\begin{equation}
\lim_{y\to0^{+}}y{}^{1-2s}\partial_{y}\left(\nabla_{x}\cdot(A(x)\nabla_{x}U)\right)=0.\label{eq:Limit 1}
\end{equation}
Then \eqref{eq:Key ingredient of iteration} will imply that 
\begin{equation}
\lim_{y\to0^{+}}y^{-2s}\nabla_{x}\cdot(A(x)\nabla_{x}U)=0.\label{eq:Limit 2}
\end{equation}
Recall that $U$ satisfies the equation 
\begin{equation}
\nabla_{x,y}\cdot(y{}^{1-2s}\widetilde{A}(x)\nabla_{x,y}U)=0\mbox{ in }B_{+}^{n+1}(x_{0},R),\label{nonreflection equation}
\end{equation}
or equivalently, $U$ fulfills 
\[
\partial_{y}(y^{1-2s}\partial_{y}U)=-y^{1-2s}\nabla_{x}\cdot(A(x)\nabla_{x}U).
\]
By using \eqref{eq:Limit 1}, we have 
\[
\lim_{y\to0^{+}}\partial_{y}(y^{1-2s}\partial_{y}U)=0.
\]
Next, recall that $\lim_{y\to0}y^{1-2s}\partial_{y}U=d_{s}\mathcal{L}^{s}u=0$
for some constant $d_{s}$ and use \eqref{eq:Key ingredient of iteration}
again, then we obtain 
\[
\lim_{y\to0^{+}}y^{-2s}\partial_{y}U=\lim_{y\to0^{+}}y^{-2s-1}U=0.
\]

\textbf{}\\
\textbf{Step 2. Iteration}

Let us differentiate \eqref{nonreflection equation} with respect
to $y$ and consider $U$ to be a weak solution of 
\begin{eqnarray}
\partial_{y}^{2}(y^{1-2s}\partial_{y}U) & = & -(1-2s)y^{-2s}\nabla_{x}\cdot(A(x)\nabla_{x}U)\label{eq:one more derivative}\\
 &  & -y^{1-2s}\partial_{y}\nabla_{x}\cdot(A(x)\nabla_{x}U)\mbox{ in }B_{+}^{n+1}(x_{0},R)\nonumber 
\end{eqnarray}
with 
\[
\lim_{y\to0^{+}}\partial_{y}(y^{1-2s}\partial_{y}U)=0\mbox{ on }B_{+}^{*}(x_{0},R).
\]
Plug \eqref{eq:Limit 1} and \eqref{eq:Limit 2} into \eqref{eq:one more derivative},
we have 
\[
\lim_{y\to0^{+}}\partial_{y}^{2}(y^{1-2s}\partial_{y}U)=\lim_{y\to0^{+}}y^{-2s-2}U=0.
\]
As previous arguments, let us take the function $W(x,y)=y^{1-2s}\partial_{y}U(x,y)$
with $\lim_{y\to0^{+}}W(x,y)=0$, then we can reflect the function
$W$ to be $\widetilde{W}(x,y)$ into a whole ball in $\mathbb{R}^{n+1}$(see
\ref{W equation in SUCP}). Since $U(x,y)$ is $C^{\infty}$-smooth
in the $x$-direction, so is $\widetilde{W}(x,y)$. Therefore, we
can differentiate $\widetilde{W}(x,y)$ with respect to $x$-variables
with arbitrary order. Then by repeating Step 1, we will obtain the
continuity of $\partial_{y}\left(\nabla_{x}\cdot(A(x)\nabla W)\right)$
and 
\[
\lim_{y\to0}\partial_{y}\left(\nabla_{x}\cdot(A(x)\nabla W)\right)=0.
\]
To sum up, after these iterate procedures and use the $x$-direction
derivatives, then we can get
\begin{eqnarray*}
\lim_{y\to0^{+}}\partial_{y}\left(y^{1-2s}\partial_{y}(\nabla_{x}\cdot(A(x)\nabla_{x}U))\right) & = & \lim_{y\to0^{+}}y^{-2s}\partial_{y}(\nabla_{x}\cdot(A(x)\nabla_{x}U))\\
 & = & \lim_{y\to0^{+}}y^{-2s-1}(\nabla_{x}\cdot(A(x)\nabla_{x}U))=0.
\end{eqnarray*}
Note that the right hand sides of these terms is obtained by differentiating
(\ref{eq:one more derivative}) with $y$ direction (in the weak sense)
and they may  involve higher order derivatives with respect to $x$-variables,
hence, we can use the bootstrap arguments to proceed previous arguments.

\textbf{}\\
\textbf{Step 3. Conclusion}

By using the bootstrap arguments, we can get 
\begin{equation}
\lim_{y\to0^{+}}y^{-m}U(x,y)=0\mbox{ for all }m\in\mathbb{N}\mbox{ and }x\in B^{*}(x_{0},R),\label{vanishing to infinitie order at y}
\end{equation}
which implies that $U$ vanishes to infinite order in the $y$-direction
on the plane $\partial\mathbb{R}_{+}^{n+1}$, and in the tangential
$x$-direction it is zero on the plane $\partial\mathbb{R}_{+}^{n+1}$
and this proves the theorem.\end{proof}
\begin{cor}
Let $U$ be the same function as in Theorem \ref{Theorem vanishing to infinite order}
and $\widetilde{U}$ be the even reflection of $U$ given by (\ref{eq:even reflection}),
then the function $\widetilde{U}$ vanishes to infinite order on $B^{*}(x_{0},R)$.\end{cor}
\begin{proof}
Since $\widetilde{U}$ is an even reflection of $U$, then we can
repeat the same proof as in Theorem \ref{Theorem vanishing to infinite order}
in the lower half space $\mathbb{R}^{n+1}\cap\{y<0\}$, then we have
\begin{equation}
\lim_{y\to0^{-}}(-y)^{-m}U(x,-y)=0\mbox{ for all }m\in\mathbb{N}\mbox{ and }x\in B^{*}(x_{0},R).\label{eq:vanishing to infinite order from negative part}
\end{equation}
Combining \eqref{vanishing to infinitie order at y} and \eqref{eq:vanishing to infinite order from negative part},
we obtain that 
\begin{equation}
\lim_{y\to0}|y|^{-m}\widetilde{U}(x,y)=0\mbox{ for all }m\in\mathbb{N}\mbox{ and }x\in B^{*}(x_{0},R),\label{vanishing to infinite order for |y|}
\end{equation}
which completes the proof.\end{proof}
\begin{prop}
\label{Prop SUCP}\cite[Corollary 3.9]{yu2016unique} Let $\widetilde{U}\in H^{1}(B^{n+1}(x_{0},1),|y|^{1-2s})$
be a solution to 
\begin{equation}
\nabla_{x,y}\cdot(|y|^{1-2s}\widetilde{A}(x)\nabla_{x,y}\widetilde{U})=0\mbox{ in }B^{n+1}(0,1).\label{eq:SUCP equation}
\end{equation}
Then the equation \eqref{eq:SUCP equation} possesses the SUCP, whenever
$A(x)$ satisfies the hypothesis ($\mathcal{H}$). 
\end{prop}
Recall that the equation \eqref{eq:SUCP equation} has the SUCP if
$\widetilde{U}\in H^{1}(B^{n+1}(0,1),|y|^{1-2s})$ is a weak solution
of \eqref{eq:SUCP equation} and $\widetilde{U}$ vanishes to infinite
order, then $\widetilde{U}\equiv0$ in $B^{n+1}(0,1)$.
\begin{proof}[Proof of Proposition \ref{Prop SUCP}]
 Firstly, the condition of vanishing to infinite order \eqref{vanishing to infinite order for |y|}
shows that $U\in L^{2}(B^{*}(x_{0},R)\times(-r_{0},r_{0}),|y|^{1-2s})$
for $R,r_{0}\ll1$, since 
\begin{equation}
\int_{B^{*}(x_{0},R)\times(-r_{0},r_{0})}|y|^{1-2s}|\widetilde{U}|^{2}\,dxdy\leq\int_{B^{*}(x_{0},R)\times(-r_{0},r_{0})}|y|^{-m}|\widetilde{U}|^{2}\,dxdy<1\label{eq:VTF}
\end{equation}
for $y\leq r_{0}\ll1$ being sufficiently small enough and for any
$m\in\mathbb{N}$ with $m\geq2$.

On the other hand, by using the doubling inequality \eqref{Doubling inequality}
in Appendix, we have 
\begin{equation}
\int_{B^{n+1}(x_{0},1)}|y|^{1-2s}|\widetilde{U}|^{2}dxdy\leq C\int_{B^{n+1}(x_{0},\frac{1}{2})}|y|^{1-2s}|\widetilde{U}|^{2}dxdy\label{eq:doubling 2}
\end{equation}
where the constant $C$ same as in \eqref{Doubling inequality}. Now,
by iterating \eqref{eq:doubling 2}, then we have 

\begin{align*}
 & \int_{B^{n+1}(x_{0},1)}|y|^{1-2s}|\widetilde{U}|^{2}dxdy\\
\leq & C^{N}\int_{B^{n+1}(x_{0},\frac{1}{2^{N}})}|y|^{1-2s}|\widetilde{U}|^{2}dxdy\\
\leq & C^{N}(\frac{1}{2})^{N(m-1)}\int_{B^{n+1}(x_{0},\frac{1}{2^{N}})}|y|^{2-2s-m}|\widetilde{U}|^{2}dxdy\\
\leq & C^{N}(\frac{1}{2})^{N(m-1)}\int_{B(x_{0},R)\times(-r_{0},r_{0})}|y|^{-m}|\widetilde{U}|^{2}dxdy,
\end{align*}
for $N$ large such that $\dfrac{1}{2^{N}}<\min\left\{ R,r_{0}\right\} $
and for any $m\in\mathbb{N}$ with $m\geq2$. Now, since $\widetilde{U}$
has vanishing order at $x_{0}$, by using \eqref{eq:VTF} $\int_{B(x_{0},R)\times(-r_{0},r_{0})}|y|^{-m}|\widetilde{U}|^{2}dxdy$
remains bounded and $C^{N}(\dfrac{1}{2})^{Nm}\to0$ as $m\to\infty$.
This implies $\widetilde{U}=0$ in $B^{n+1}(x_{0},1)$, which completes
the proof.\end{proof}
\begin{lem}
\label{Lemma with zero}Let $u\in H^{s}(\mathbb{R}^{n})$, if $u=\mathcal{L}^{s}u=0$
in any ball $B(x_{0},R)\subseteq\mathbb{R}^{n}$, then $U=0$ in $\overline{B_{+}^{n+1}(x_{0},R)}$,
where $U$ is the function in Theorem \ref{Theorem vanishing to infinite order}.\end{lem}
\begin{proof}
As $u\in H^{s}(\mathbb{R}^{n})$ satisfying $u|_{B(x_{0},R)}=\mathcal{L}{}^{s}u|_{B(x_{0},R)}=0$,
so from Theorem \ref{Theorem vanishing to infinite order}, we have
$U$ vanishes infinite order on $\partial\mathbb{R}_{+}^{n+1}$ so
does $\widetilde{U}$, where $\widetilde{U}$ is the even reflection
of $U$ defined by \eqref{eq:even reflection}. Therefore, by using
Proposition \ref{Prop SUCP}, we have SUCP for \eqref{eq:SUCP equation}.
Consequently it follows $\widetilde{U}=0$ in $B^{n+1}(x_{0},R)$
so $U=0$ in $\overline{B_{+}^{n+1}(x_{0},R)}$. 
\end{proof}

\subsection{Proof of Theorem \ref{thm:(Approximation-theorem)}}

Let us begin to prove Theorem \ref{thm:(Approximation-theorem)} by
using Lemma \ref{Lemma with zero}. 
\begin{proof}[Proof of Theorem \ref{thm:(Approximation-theorem)}]
 We have already shown that $U=0$ in $\overline{B_{+}^{n+1}(x_{0},R)}$.
Now, we will show $U=0$ in $\mathbb{R}_{+}^{n+1}\setminus\overline{B_{+}^{n+1}(x_{0},R)}$
also. Let us divide the case in two parts:

\textbf{Case 1.} $s\geq\dfrac{1}{2}$. Let us consider the the region
$D_{\epsilon}=\{(x,y)\,:\,x\in\mathbb{R}^{n}\mbox{ and }\epsilon<y<1/\epsilon\}$
for any $\epsilon>0$. Since the weight $y^{1-2s}$ is smooth and
positive in $\overline{D_{\epsilon}}$, thus $U$ can be realized
as a solution of a uniformly elliptic equation 
\begin{equation}
\nabla_{x,y}\cdot(y^{1-2s}\widetilde{A}(x)\nabla_{x,y}U)=0\mbox{ in }\mathbb{R}_{+}^{n+1}\label{eq:Proof of Thm 1.2}
\end{equation}
in $H^{1}(D_{\epsilon})$. Since $U$ also vanishes in $\overline{B_{+}^{n+1}(x_{0},R)}\cap D_{\epsilon}$,
where $\epsilon>0$ is chosen so small that this set is nonempty,
it follows by standard weak unique continuation property for the uniform
elliptic equation in a strip domain that $U$ has to vanish in entire
$D_{\epsilon}$. Since this is true for any $\epsilon>0$ small, one
has $U=0$ in $\mathbb{R}_{+}^{n+1}$ as required. Hence as a trace
of $U\in H^{1}(\mathbb{R}_{+}^{n+1},y^{1-2s})$, $U(x,0)=u(x)=0$
in $\mathbb{R}^{n}$.

\textbf{Case 2.} $s<\dfrac{1}{2}$. In order to establish our claim,
$U=0$ in $\mathbb{R}_{+}^{n+1}$ in this case, we write 
\begin{align}
 & \int_{\mathbb{R}_{+}^{n+1}}y^{1-2s}\widetilde{A}\nabla_{x,y}U\cdot\nabla_{x,y}Udx\,dy\nonumber \\
= & \lim_{R\to\infty}\int_{B_{+}^{n+1}(0,R)}y^{1-2s}\widetilde{A}\nabla_{x,y}U\cdot\nabla_{x,y}Udx\,dy\nonumber \\
= & \lim_{R\to\infty}\int_{\partial B_{+}^{n+1}(0;1,R)}y^{1-2s}(\widetilde{A}\nabla_{x,y}U\cdot\nu)UdS(x,y),\label{eq:limit1}
\end{align}
where $\partial B_{+}^{n+1}(0;1,R)=\partial B_{+}^{n+1}(0,R)\cup\partial B_{+}^{n+1}(0,1)\cup B^{0}(0;1,R)$
and $B^{0}(0;1,R)=\left\{ (x,0)\in\mathbb{R}^{n+1};1\leq|x|\leq R\right\} $.
Then using the fact $w=0$ on $\partial B_{+}^{n+1}(0,1)$ and $s<\dfrac{1}{2}$
gives the integrand in \eqref{eq:limit1} to be 0 on $B^{0}(0;1,R)$.
Hence, 
\begin{align}
 & \int_{\mathbb{R}_{+}^{n+1}}y^{1-2s}\widetilde{A}\nabla_{x,y}U\cdot\nabla_{x,y}Udx\,dy\nonumber \\
= & \lim_{R\to\infty}\int_{\partial B_{+}^{n+1}(0,R)}y^{1-2s}(\widetilde{A}\nabla_{x,y}U\cdot\nu)UdS(x,y)\nonumber \\
= & \lim_{R\to\infty}\int_{\mathbb{R}_{+}^{n+1}\backslash B_{+}^{n+1}(0,R)}y^{1-2s}\widetilde{A}\nabla_{x,y}U\cdot\nabla_{x,y}Udx\,dy=0\label{eq:limit}
\end{align}
since $U\in H^{1}(\mathbb{R}_{+}^{n+1},y^{1-2s})$. Thus, it follows
from \eqref{eq:Proof of Thm 1.2} and \eqref{eq:limit}, $U\equiv0$
in $\mathbb{R}_{+}^{n+1}$ and consequently, as a trace of $U\in H^{1}(\mathbb{R}_{+}^{n+1},y^{1-2s})$,
$U(x,0)=u(x)=0$ in $\mathbb{R}^{n}$. This completes the proof of
Theorem 1.2 for $H^{s}(\mathbb{R}^{n})$ class of functions. 
\end{proof}

\subsection{Runge approximation property}

We will utilize the Runge approximation property for solutions of
variable coefficients fractional operators. Recall that 
\[
X=H^{s}(\mathbb{R}^{n})/\widetilde{H}^{s}(\Omega)
\]
be a quotient space and if $q\in L^{\infty}(\Omega)$ satisfies the
eigenvalue condition \eqref{eq:eigenvalue condition}, we define the
operator $P_{q}$ by 
\begin{equation}
P_{q}:H^{s}(\Omega_{e})\to H^{s}(\mathbb{R}^{n}),\,f\mapsto u,\label{poisson_operator_definition}
\end{equation}
where $u\in H^{s}(\mathbb{R}^{n})$ is the unique solution of $(\mathcal{L}^{s}+q)u=0$
in $\Omega$ with $u-f\in\widetilde{H}^{s}(\Omega)$.
\begin{lem}
\label{lem:Approximation Lemma}Let $\Omega\subseteq\mathbb{R}^{n}$
be bounded open set with Lipschitz boundary and $A(x)$ be matrix-valued
function defined in $\mathbb{R}^{n}$ satisfying the hypothesis ($\mathcal{H}$).
Assume that $s\in(0,1)$ and $q\in L^{\infty}(\Omega)$ satisfy the
eigenvalue condition \eqref{eq:eigenvalue condition}. Let $\mathcal{O}$
be any open subset of $\Omega_{e}$. Consider the set 
\begin{align*}
\mathbb{D}=\{u|_{\Omega}\,;\,u=P_{q}f,\ f\in C_{c}^{\infty}(\mathcal{O})\}.
\end{align*}
Then $\mathbb{D}$ is dense in $L^{2}(\Omega)$.\end{lem}
\begin{proof}
By the Hahn-Banach theorem, it is only need to show that for any $v\in L^{2}(\Omega)$
satisfying $(v,w)_{\Omega}=0$ for any $w\in\mathbb{D}$, then $v\equiv0$.
Let $v$ be a such function, which means $v$ satisfies 
\begin{equation}
(v,r_{\Omega}P_{q}f)=0,\mbox{ for any }f\in C_{c}^{\infty}(\mathcal{O}).\label{eq:111111}
\end{equation}
Now, let $\phi\in\widetilde{H^{s}}(\Omega)$ be the solution of $(\mathcal{L}^{s}+q)\phi=v$
in $\Omega$. We want to show that for any $f\in C_{c}^{\infty}(\mathcal{O})$,
the following relation 
\begin{equation}
\mathcal{B}_{q}(\phi,f)=-(v,r_{\Omega}P_{q}f)_{\Omega}\label{eq:formal adjoint relation}
\end{equation}
holds. In other words, $\mathcal{B}_{q}(\phi,w)=(v,r_{\Omega}w)$
for any $w\in\widetilde{H^{s}}(\Omega)$. To prove \eqref{eq:formal adjoint relation},
we denote $u_{f}=P_{q}f\in H^{s}(\mathbb{R}^{n})$ with $f\in C_{c}^{\infty}(\mathcal{O})$
such that $f-u_{f}\in\widetilde{H^{s}}(\Omega)$, then we have 
\[
\mathcal{B}_{q}(\phi,f)=\mathcal{B}_{q}(\phi,f-u_{f})=(v,r_{\Omega}(f-u_{f}))_{\Omega}=-(v,r_{\Omega}P_{q}f)_{\Omega},
\]
where we have used the facts that $u_{f}$ is a solution and $\phi\in\widetilde{H^{s}}(\Omega)$.
Note that \eqref{eq:111111} and \eqref{eq:formal adjoint relation}
imply that 
\[
\mathcal{B}_{q}(\phi,f)=0\mbox{ for any }f\in C_{c}^{\infty}(\mathcal{O}).
\]
Moreover, we know that $r_{\Omega}f=0$ because $f\in C_{c}^{\infty}(\mathcal{O})$
and we can derive 
\[
(\mathcal{L}^{s}\phi,f)_{\mathbb{R}^{n}}=0\mbox{ for any }f\in C_{c}^{\infty}(\mathcal{O}).
\]
In the end, we know that $\phi\in H^{s}(\mathbb{R}^{n})$ which satisfies
\[
\phi|_{\mathcal{O}}=\mathcal{L}^{s}\phi|_{\mathcal{O}}=0.
\]
By Theorem \ref{thm:(Approximation-theorem)}, we obtain $\phi\equiv0$
and then $v\equiv0$.\end{proof}
\begin{rem}
We also refer readers to \cite{lax1956stability} for more details
of the Runge approximation property for the (local) differential equations. 
\end{rem}

\section{Proof of Theorem \ref{thm: Main}\label{Section 6}}

Now, we are ready to prove the global uniqueness result for variable
coefficients fractional operators. Even though the proof is similar
as the proof in \cite{ghosh2016calder}, we still give a proof for
the completeness.
\begin{proof}[Proof of Theorem \eqref{thm: Main}]
 If $\Lambda_{q_{1}}g|_{\mathcal{O}_{2}}=\Lambda_{q_{2}}g|_{\mathcal{O}_{2}}$
for any $g\in C_{c}^{\infty}(\mathcal{O}_{1})$, where $\mathcal{O}_{1}$
and $\mathcal{O}_{2}$ are open subsets of $\Omega_{e}$, by the integral
identity in Lemma \ref{lem:(Integral-identity)-Let}, we have 
\[
\int_{\Omega}(q_{1}-q_{2})u_{1}u_{2}dx=0
\]
where $u_{1},u_{2}\in H^{s}(\mathbb{R}^{n})$ solve $(\mathcal{L}^{s}+q_{1})u_{1}=0$
and $(\mathcal{L}^{s}+q_{2})u_{2}=0$ in $\Omega$ with $u_{1}$,
$u_{2}$ having exterior values $g_{j}\in C_{c}^{\infty}(\mathcal{O}_{j})$,
for $j=1,2$.

Let $f\in L^{2}(\Omega)$, and use the approximation lemma \ref{lem:Approximation Lemma},
then there exist two sequences $(u_{j}^{1})$, $(u_{j}^{2})$ of functions
in $H^{s}(\mathbb{R}^{n})$ that satisfy 
\begin{align*}
 & (\mathcal{L}^{s}+q_{1})u_{j}^{1}=(\mathcal{L}^{s}+q_{2})u_{j}^{2}=0\text{ in \ensuremath{\Omega}},\\
 & \mbox{supp}(u_{j}^{1})\subseteq\overline{\Omega_{1}}\mbox{ and }\mbox{supp}(u_{j}^{2})\subseteq\overline{\Omega_{2}},\\
 & r_{\Omega}u_{j}^{1}=f+r_{j}^{1},\ \ r_{\Omega}u_{j}^{2}=1+r_{j}^{2},
\end{align*}
where $\Omega_{1}$, $\Omega_{2}$ are two open subsets of $\mathbb{R}^{n}$
containing $\Omega$, and $r_{j}^{1},r_{j}^{2}\to0$ in $L^{2}(\Omega)$
as $j\to\infty$. Plug these solutions into the integral identity
and pass the limit as $j\to\infty$, then we infer that 
\[
\int_{\Omega}(q_{1}-q_{2})fdx=0.
\]
Since $f\in L^{2}(\Omega)$ was arbitrary, we conclude that $q_{1}=q_{2}$. 
\end{proof}

\section{Appendix}

At the end of this paper, we present some required materials to complete
our paper.

\subsection{Stability result for the degenerate problem}

In general, we have the following result.
\begin{lem}
\label{Appendix Lemma 1}Let $h$ be a vector-valued function satisfying
$\dfrac{G}{y^{1-2s}}\in L^{2}(\mathbb{R}_{+}^{n+1},y^{1-2s})$, then
the following Dirichlet boundary value problem 
\begin{equation}
\begin{cases}
\nabla_{x,y}\cdot(y^{1-2s}\widetilde{A}(x)\nabla_{x,y}V)=\nabla_{x,y}\cdot G\mbox{ in }\mathbb{R}^{n+1},\\
V(x,0)=0\mbox{ on }\mathbb{R}^{n}
\end{cases}\label{New degenerate elliptic problem}
\end{equation}
has a unique weak solution in $H^{1}(\mathbb{R}_{+}^{n+1},y^{1-2s})$
satisfying 
\begin{equation}
\|V\|_{H^{1}(\mathbb{R}_{+}^{n+1},y^{1-2s})}\leq C\|y^{-1+2s}G\|_{L^{2}(\mathbb{R}_{+}^{n+1},y^{1-2s})},\label{eq:stabilityA}
\end{equation}
where the constant $C>0$ is independent of $G$ and $V$. 
\end{lem}
By the weak solution of \eqref{New degenerate elliptic problem} we
mean $V\in H^{1}(\mathbb{R}_{+}^{n+1},y^{1-2s})$ solves 
\begin{equation}
\int_{\mathbb{R}_{+}^{n+1}}y^{1-2s}\widetilde{A}(x)\nabla_{x,y}V\cdot\nabla_{x,y}\Psi\,dxdy=\int_{\mathbb{R}_{+}^{n+1}}y^{-1+2s}G\cdot y^{1-2s}\nabla\Psi\,dxdy,\label{eq:weakA}
\end{equation}
for all $\Psi\in H_{0}^{1}(\mathbb{R}_{+}^{n+1},y^{1-2s})$.
\begin{proof}[Proof of Lemma \ref{Appendix Lemma 1}]
Let us consider the Dirichlet functional $J:H_{0}^{1}(\mathbb{R}_{+}^{n+1},y^{1-2s})\to\mathbb{R}_{+}$
as 
\begin{equation}
J(\Psi):=\int_{\mathbb{R}_{+}^{n+1}}y^{1-2s}\widetilde{A}(x)\nabla_{x,y}\Psi\cdot\nabla_{x,y}\Psi\,dxdy-\int_{\mathbb{R}_{+}^{n+1}}y^{-1+2s}G\cdot y^{1-2s}\nabla\Psi,\label{Energy functional}
\end{equation}
If $V\in H_{0}^{1}(\mathbb{R}_{+}^{n+1},y^{1-2s})$ is an extremum
of $J(\Psi)$ in $H_{0}^{1}(\mathbb{R}_{+}^{n+1},y^{1-2s})$, then
for any $\Psi\in C_{c}^{\infty}(\mathbb{R}_{+}^{n+1})$, as a function
of $\eta$, 
\[
F(\eta):=J(V+\eta\Psi)
\]
attains its extremum at $\eta=0$ and hence $F^{\prime}(0)=0$ as
\[
\begin{split}F^{\prime}(0)= & \lim_{\eta\to0}\dfrac{J(V+\eta\Psi)-J(V)}{\eta}\\
= & 2\int_{\mathbb{R}_{+}^{n+1}}y^{1-2s}\widetilde{A}(x)\nabla_{x,y}V\cdot\nabla_{x,y}\Psi\,dxdy\\
 & -2\int_{\mathbb{R}_{+}^{n+1}}y^{-1+2s}h\cdot y^{1-2s}\nabla\Psi\,dxdy\\
= & 0,
\end{split}
\]
which gives the definition of the weak solution. As we can see from
the definition (\ref{Energy functional}) 
\begin{align*}
J(\Psi) & \geq\frac{1}{2}\int_{\mathbb{R}_{+}^{n+1}}y^{1-2s}|\nabla\Psi|^{2}\,dxdy-\frac{1}{2}\int_{\mathbb{R}_{+}^{n+1}}y^{-1+2s}|h|^{2}\,dxdy\\
 & \geq-\frac{1}{2}\int_{\mathbb{R}_{+}^{n+1}}y^{-1+2s}|h|^{2}\,dxdy
\end{align*}
that means $J(\Psi)$ is bounded from below in $H_{0}^{1}(\mathbb{R}_{+}^{n+1},y^{1-2s})$.

Therefore, $\inf_{H_{0}^{1}(\mathbb{R}_{+}^{n+1},y^{1-2s})}J(\Psi)$
is a finite number. Hence, there exists a minimizing sequence $\{\Psi_{k}\}_{k=1}^{\infty}\subset H_{u}^{1}(\mathbb{R}_{+}^{n+1},y^{1-2s})$
such that 
\[
\lim_{k\to\infty}J(\Psi_{k})=\inf_{H_{0}^{1}(\mathbb{R}_{+}^{n+1},y^{1-2s})}J(\Psi).
\]
Next we observe that, the functional turns out to be weakly lower
semi-continuous over its domain of definition, i.e. 
\[
J(\Psi)\leq\liminf_{k\to\infty}J(\Psi_{k}),\quad\mbox{if }\Psi_{k}\rightharpoonup\Psi\mbox{ weakly in }H_{0}^{1}(\mathbb{R}_{+}^{n+1},y^{1-2s}).
\]
This simply follows as $\mbox{if }\Psi_{k}\rightharpoonup\Psi\mbox{ weakly in }H_{0}^{1}(\mathbb{R}_{+}^{n+1},y^{1-2s})$
then 
\begin{align*}
 & \int_{\mathbb{R}_{+}^{n+1}}y^{1-2s}\widetilde{A}(x)\nabla_{x,y}\Psi\cdot\nabla_{x,y}\Psi\,dxdy\\
\leq & \liminf_{k\to\infty}\int_{\mathbb{R}_{+}^{n+1}}y^{1-2s}\widetilde{A}(x)\nabla_{x,y}\Psi_{k}\cdot\nabla_{x,y}\Psi_{k}\,dxdy.
\end{align*}
Thus, if $\{\Psi_{k}\}_{k=1}^{\infty}$ is a minimizing sequence,
i.e. , if 
\[
J(\Psi_{k})\to\inf_{H_{0}^{1}(\mathbb{R}_{+}^{n+1},y^{1-2s})}J(V)
\]
then there exists a subsequence $\{\Psi_{k_{j}}\}_{j=1}^{\infty}$
such that $\Psi_{k_{j}}\rightharpoonup V$ weakly in $H_{0}^{1}(\mathbb{R}_{+}^{n+1},y^{1-2s})$
and hence 
\[
\inf_{H_{0}^{1}(\mathbb{R}_{+}^{n+1},y^{1-2s})}J(\Psi)\leq J(V)\leq\liminf_{k\to\infty}J(\Psi_{k})=\inf_{H_{0}^{1}(\mathbb{R}_{+}^{n+1},y^{1-2s})}J(\Psi),
\]
Therefore, $J(V)=\inf_{H_{0}^{1}(\mathbb{R}_{+}^{n+1},y^{1-2s})}J(\Psi)$
and we achieve our goal.

Next, we claim that $V\in H_{0}^{1}(\mathbb{R}_{+}^{n+1},y^{1-2s})$
is the unique minimizer of $J(\Psi)$. Assume that $V_{1}$, $V_{2}\in H_{0}^{1}(\mathbb{R}_{+}^{n+1},y^{1-2s})$
are weak solutions of \eqref{New degenerate elliptic problem}, then
$V_{1}-V_{2}\in H_{0}^{1}(\mathbb{R}_{+}^{n+1},y^{1-2s})$ satisfies
the following integral identity 
\[
\int_{\mathbb{R}_{+}^{n+1}}y^{1-2s}\widetilde{A}(x)\nabla_{x,y}(V_{1}-V_{2})\cdot\nabla_{x,y}(V_{1}-V_{2})\,dxdy=0,
\]
which implies that $V_{1}=V_{2}$. This shows that \eqref{New degenerate elliptic problem}
has a unique weak solution in $H_{0}^{1}(\mathbb{R}_{+}^{n+1},y^{1-2s})$.
The remaining stability estimate \eqref{New degenerate elliptic problem}
simply follows from \eqref{eq:weakA} by taking $\Psi=V$ there, to
have 
\[
\|V\|_{H^{1}(\mathbb{R}_{+}^{n+1},y^{1-2s})}\leq C\|y^{-1+2s}G\|_{L^{2}(\mathbb{R}_{+}^{n+1},y^{1-2s})}
\]
for some constant $C>0$. \end{proof}
\begin{lem}
\label{lem:Pyintegral} Let $P_{y}^{s}$ be the Poisson kernel given
by \eqref{eq:representation of Poisson kernel}. Then 
\begin{equation}
\lim_{y\rightarrow0^{+}}\int_{\mathbb{R}^{n}}P_{y}^{s}(x,z)dz=1,\quad x\in\mathbb{R}^{n},\label{eq:Pyto0}
\end{equation}
and 
\begin{equation}
\int_{\mathbb{R}^{n}}P_{y}^{s}(x,z)dz=1,\quad x\in\mathbb{R}^{n},\,y>0.\label{eq:Py>0}
\end{equation}
\end{lem}
\begin{proof}
The limit \eqref{eq:Pyto0} is verified in \cite[Theorem 2.1]{stinga2010extension}.
We only need to show \eqref{eq:Py>0}. The following identity holds
\[
\int_{0}^{\infty}e^{-\frac{y^{2}}{4t}}\frac{e^{-b|x-z|^{2}/t}}{t^{n/2}}\frac{dt}{t^{1+s}}=\frac{c(n,s,b)}{\left(b|x-z|^{2}+y^{2}\right)^{\frac{n+2s}{2}}},
\]
such that 
\begin{equation}
\frac{c_{1}y^{2s}}{\left(b_{1}|x-z|^{2}+y^{2}\right)^{\frac{n+2s}{2}}}\leq P_{y}^{s}(x,z)\leq\frac{c_{2}y^{2s}}{\left(b_{2}|x-z|^{2}+y^{2}\right)^{\frac{n+2s}{2}}},\label{eq:Estimate for Poisson}
\end{equation}
with some positive constant $b_{j},c_{j}$, for $j=1,2$. It is obtained
by applying the estimate \eqref{eq:Estimate for Poisson} that 
\begin{eqnarray}
\lim_{\epsilon\rightarrow0^{+}}\int_{|z-x|>\epsilon}P_{y}^{s}(x,z)dz & = & \lim_{\epsilon\rightarrow0^{+}}\int_{|z-x|>\epsilon}\frac{y^{2s}}{4^{s}\Gamma(s)}\int_{0}^{\infty}e^{-\frac{y^{2}}{4t}}p_{t}(x,z)\dfrac{dt}{t^{1+s}}dz\nonumber \\
 & = & \lim_{\epsilon\rightarrow0^{+}}\int_{0}^{\infty}\int_{|z-x|>\epsilon}\frac{y^{2s}}{4^{s}\Gamma(s)}e^{-\frac{y^{2}}{4t}}p_{t}(x,z)\,dz\dfrac{dt}{t^{1+s}}\nonumber \\
 & = & \int_{0}^{\infty}\frac{y^{2s}}{4^{s}\Gamma(s)}e^{-\frac{y^{2}}{4t}}\dfrac{dt}{t^{1+s}}\label{eq:ProofPy}\\
 &  & -\lim_{\epsilon\rightarrow0^{+}}\int_{0}^{\infty}\int_{|z-x|\le\epsilon}\frac{y^{2s}}{4^{s}\Gamma(s)}e^{-\frac{y^{2}}{4t}}p_{t}(x,z)\,dz\dfrac{dt}{t^{1+s}}\nonumber \\
 & = & \,I_{1}-\frac{1}{4^{s}\Gamma(s)}\lim_{\epsilon\rightarrow0^{+}}I_{\epsilon}(y),\nonumber 
\end{eqnarray}
where we have used the fact that the heat kernel $p_{t}(x,z)$  satisfies
 $\int_{\mathbb{R}^{n}}p_{t}(x,z)dz=1$. We have from the Gamma function
that the integral $I_{1}=1$, providing $s\in(0,1)$. We claim that
$\lim_{\epsilon\rightarrow0^{+}}I_{\epsilon}(y)=0$ for any $y>0$.
In fact, by \eqref{eq:pointwise estimates for p_t} one has

\begin{eqnarray*}
I_{\epsilon}(y) & \leq & c\int_{0}^{\infty}y^{2s}e^{-\frac{y^{2}}{4t}}\int_{|z-x|\le\epsilon}e^{-b\frac{|x-z|^{2}}{t}}\,dz\,\frac{dt}{t^{1+s+n/2}}\\
 & = & c\int_{0}^{\infty}y^{2s}e^{-\frac{y^{2}}{4t}}\int_{B_{\epsilon}}e^{-b\frac{|z|^{2}}{t}}\,dz\,\frac{dt}{t^{1+s+n/2}}\\
 & = & 4\pi c\int_{0}^{\infty}y^{2s}e^{-\frac{y^{2}}{4t}}\int_{0}^{\epsilon}e^{-b\frac{r^{2}}{t}}\,dr\,\frac{dt}{t^{1+s+n/2}}\\
 & = & 4\pi c_{1}\int_{0}^{\epsilon}\frac{y^{2s}}{\left(br^{2}+y^{2}\right)^{s+n/2}}\,dr.
\end{eqnarray*}
Therefore, one can pass the limit $\epsilon\rightarrow0^{+}$ in \eqref{eq:ProofPy}
and thus obtain \eqref{eq:Py>0}. 
\end{proof}

\subsection{Almgren's type frequency function and the doubling inequality for
the degenerate problem}

\label{SUCP} Here we mention the strong unique continuation property
for the degenerate problem $\nabla_{x,y}\cdot(|y|^{1-2s}\widetilde{A}(x)\nabla_{x,y}\widetilde{U})=0$
in $B^{n+1}(0,1)$. The proof relies on the technique in using the
Almgren's frequency function method, which was introduced by Yu \cite{yu2016unique}.

To simplify the notation, let us denote $B_{r}^{n+1}:=B^{n+1}(0,r)$
and $z=(x,y)\in\mathbb{R}^{n+1}$. For $z\neq0$, we define 
\[
\mu(z):=\dfrac{\left(\widetilde{A}(z)z\right)\cdot z}{|z|^{2}}\in\mathbb{R}\mbox{ and }\overrightarrow{\beta}(z):=\dfrac{\widetilde{A}(z)z}{\mu(z)}\in\mathbb{R}^{n+1},
\]
then from the ellipticity condition (\ref{eq:ellipticity and symmetry condition}),
it is easy to see that 
\[
\widetilde{\Lambda}^{-1}\leq\mu(z)\leq\widetilde{\Lambda}\mbox{ and }|\overrightarrow{\beta}(z)|\leq\widetilde{\Lambda}|z|\mbox{ for all }z\in\mathbb{R}^{n+1}
\]
for some universal constant $\widetilde{\Lambda}>0$. In addition,
by the standard coordinates transformation technique, we may assume
that $\widetilde{A}(0)=I_{n+1}$, which is an $(n+1)\times(n+1)$
identity matrix, then we have the following estimates hold for $\mu(z)$
and $\overrightarrow{\beta}(z)=(\beta_{1}(z),\beta_{2}(z),\cdots,\beta_{n+1}(z))$:
\begin{equation}
\left|\dfrac{\partial}{\partial r}\mu(rz)\right|\leq C\mbox{ for }r>0\mbox{ and }\dfrac{\partial\beta_{i}}{\partial z_{j}}(z)=\delta_{ij}+O(|z|),\label{Estimates for mu and beta}
\end{equation}
where $\delta_{ij}$ is the Kronecker delta and the constant $C>0$
depends on $\widetilde{A}(z)=(\widetilde{a}_{jk}(z))_{j,k=1}^{n+1}$.
The estimates (\ref{Estimates for mu and beta}) were proved in \cite{tao2008weighted,yu2016unique},
so we skip the details.

Let $\widetilde{U}\in H^{1}(\mathbb{R}_{+}^{n+1},|y|^{1-2s})$ and
consider 
\begin{eqnarray}
H(r) & := & \int_{\partial B_{r}^{n+1}}|y|^{1-2s}\mu(z)|\widetilde{U}(z)|^{2}dS(z),\label{eq:H(r)}\\
D(r) & := & \int_{B_{r}^{n+1}}|y|^{1-2s}\left(\widetilde{A}(z)\nabla\widetilde{U}\right)\cdot\nabla\widetilde{U}dz,\label{eq:D(r)}
\end{eqnarray}
where $\nabla:=\nabla_{z}=\nabla_{x,y}$ in $\mathbb{R}^{n+1}$ and
it is easy to see that $H(r)$ exists for almost every $r>0$ as a
surface integral, since the volume integral ($\int_{0}^{R}H(r)dr<\infty$)
exists due to $\widetilde{U}\in H^{1}(\mathbb{R}_{+}^{n+1},|y|^{1-2s})$).
Next, similar to \cite{lin1991nodal,tao2008weighted,yu2016unique},
we define the corresponding \textit{Almgren's frequency function}
by 
\[
N(r):=\dfrac{rD(r)}{H(r)},
\]
and we have the following lemmas.
\begin{lem}
For any $r\in(0,1)$, $H(r)=0$ whenever $\widetilde{U}\equiv0$ in
$B_{r}^{n+1}$. \end{lem}
\begin{proof}
If $H(r)=0$, it implies that $\widetilde{U}=0$ on $\partial B_{r}^{n+1}$.
Hence, by the uniqueness of the solution of the degenerate problem
(for example, see\cite{fabes1982local}), we conclude $\widetilde{U}\equiv0$
in $B_{r}^{n+1}$. \end{proof}
\begin{lem}
The function $H(r)$ is differentiable and 
\begin{eqnarray}
H'(r) & = & \left(\dfrac{(n+1-2s)}{r}+O(1)\right)H(r)+2D(r).\label{eq:H'(r)}
\end{eqnarray}
\end{lem}
\begin{proof}
By change of variables, we have 
\begin{eqnarray*}
H(r) & = & \int_{\partial B_{r}^{n+1}}|y|^{1-2s}\mu(z)|\widetilde{U}(z)|^{2}dS\\
 & = & r^{n+1-2s}\int_{\partial B_{1}^{n+1}}|y|^{1-2s}\mu(rz)|\widetilde{U}(rz)|^{2}dS,
\end{eqnarray*}
then 
\begin{eqnarray*}
H'(r) & = & \dfrac{d}{dr}H(r)\\
 & = & (n+1-2s)r^{n-2s}\int_{\partial B_{1}^{n+1}}|y|^{1-2s}\mu(rz)|\widetilde{U}(rz)|^{2}dS\\
 &  & +r^{n+1-2s}\int_{\partial B_{1}^{n+1}}|y|^{1-2s}\dfrac{\partial}{\partial r}\mu(rz)|\widetilde{U}(rz)|^{2}dS\\
 &  & +2r^{n+1-2s}\int_{\partial B_{1}^{n+1}}|y|^{1-2s}\mu(rz)\widetilde{U}(rz)\dfrac{\partial}{\partial r}\widetilde{U}(rz)dS,
\end{eqnarray*}
Note that $H'(r)$ exists for a.e. $r>0$ due to $\widetilde{U}\in H^{1}(\mathbb{R}_{+}^{n+1},|y|^{1-2s})$
and $\dfrac{\partial}{\partial r}\mu(rz)$ is bounded by constant
$C>0$ (see (\ref{Estimates for mu and beta})) and after change of
variables back, we obtain 
\begin{eqnarray*}
H'(r) & \leq & \dfrac{(n+1-2s)}{r}\int_{\partial B_{r}^{n+1}}|y|^{1-2s}\mu(z)|\widetilde{U}(z)|^{2}dS\\
 &  & +C\int_{\partial B_{r}^{n+1}}|y|^{1-2s}|\widetilde{U}(z)|^{2}dS\\
 &  & +2\int_{\partial B_{r}^{n+1}}|y|^{1-2s}\mu(z)\widetilde{U}(z)\dfrac{\partial\widetilde{U}}{\partial\nu}(z)dS,
\end{eqnarray*}
where $\nu$ is a unit outer normal on $\partial B_{1}^{n+1}$. By
using the regularity assumption for $A(x)$ and $\widetilde{U}\in H^{1}(\mathbb{R}_{+}^{n+1},|y|^{1-2s})$,
we have $C\int_{|z|=r}|y|^{1-2s}|\widetilde{U}(z)|^{2}dS$ bounded
for a.e. $r>0$. Therefore, we have 
\begin{eqnarray*}
H'(r) & = & \left(\dfrac{(n+1-2s)}{r}+O(1)\right)H(r)\\
 &  & +2\int_{\partial B_{r}^{n+1}}|y|^{1-2s}\mu(z)\widetilde{U}(z)\dfrac{\partial\widetilde{U}}{\partial\nu}(z)dS.
\end{eqnarray*}

Finally, we will show that 
\begin{equation}
\int_{\partial B_{r}^{n+1}}|y|^{1-2s}\mu(z)\widetilde{U}(z)\dfrac{\partial\widetilde{U}}{\partial\nu}(z)dS=D(r)+O(1)H(r).\label{eq:another represetation for H'(r)}
\end{equation}
By using the equation $\nabla\cdot(|y|^{1-2s}\widetilde{A}\nabla\widetilde{U})=0$,
we can rewrite $D(r)$ in terms of 
\[
D(r)=\int_{B_{r}^{n+1}}\nabla\cdot(|y|^{1-2s}\widetilde{U}\widetilde{A}\nabla\widetilde{U})dz=\int_{\partial B_{r}^{n+1}}|y|^{1-2s}\widetilde{U}(\widetilde{A}\nu)\cdot\nabla\widetilde{U}dS.
\]
We define $\mathcal{T}(z):=\widetilde{A}\nu-\mu(z)\nu\in\mathbb{R}^{n+1}$
and note that 
\[
\mathcal{T}\cdot\nu=(\widetilde{A}\nu-\mu(z)\nu)\cdot\nu=0\mbox{ on }\partial B_{r}^{n+1},
\]
which means $\mathcal{T}(z)$ is a tangential vector of $\partial B_{r}^{n+1}$.
From the divergence theorem on $\partial B_{r}^{n+1}$, we can derive
that 
\begin{align*}
 & D(r)-\int_{\partial B_{r}^{n+1}}|y|^{1-2s}\mu(z)\widetilde{U}\dfrac{\partial\widetilde{U}}{\partial\nu}dS\\
= & \int_{\partial B_{r}^{n+1}}|y|^{1-2s}\widetilde{U}\nabla\widetilde{U}\cdot(\widetilde{A}\nu-\mu(z)\nu)dS\\
= & -\dfrac{1}{2}\int_{\partial B_{r}^{n+1}}|y|^{1-2s}|\widetilde{U}|^{2}\nabla\cdot\mathcal{T}dS-\dfrac{1}{2}\int_{\partial B_{r}^{n+1}}|\widetilde{U}|^{2}\left(\nabla|y|^{1-2s}\right)\cdot\mathcal{T}dS.
\end{align*}
From direct computation, we have $\left|\nabla_{x}\cdot\mathcal{T}\right|\leq C_{n,A}$
for some constant $C_{n,A}>0$ depending on $n$ and $A(x)$ and then
\begin{equation}
\int_{\partial B_{r}^{n+1}}|y|^{1-2s}|\widetilde{U}|^{2}\nabla_{x}\cdot\mathcal{T}dS=O(1)H(r).\label{Divergent 1}
\end{equation}
On the other hand, it is not hard to see that 
\[
\left|\int_{\partial B_{r}^{n+1}}|\widetilde{U}|^{2}\left(\nabla|y|^{1-2s}\right)\cdot\mathcal{T}dS\right|\leq\dfrac{1}{r}\int_{\partial B_{r}^{n+1}}|\widetilde{U}|^{2}\left|(1-2s)y|y|^{-2s}\left(1-\mu(z)\right)\right|dS
\]
and by using $\left|1-\mu(z)\right|\leq C_{A}|z|$, for some constant
$C_{A}>0$, then we can derive 
\[
\left|\int_{\partial B_{r}^{n+1}}|\widetilde{U}|^{2}\left(\nabla|y|^{1-2s}\right)\cdot\mathcal{T}dS\right|\leq C\int_{\partial B_{r}^{n+1}}|y|^{1-2s}|\widetilde{U}|^{2}dS=O(1)H(r).
\]
This proves the lemma.\end{proof}
\begin{lem}
The function $D(r)$ is differentiable with 
\begin{equation}
D'(r)=\left(\dfrac{n-2s}{r}+O(1)\right)D(r)+2\int_{\partial B_{r}^{n+1}}|y|^{1-2s}\dfrac{1}{\mu}\left|(\widetilde{A}\nu)\cdot\nabla\widetilde{U}\right|^{2}dS.\label{eq:D'(r)}
\end{equation}
\end{lem}
\begin{proof}
It is easy to see that 
\[
D'(r)=\int_{\partial B_{r}^{n+1}}|y|^{1-2s}\widetilde{A}(z)\nabla\widetilde{U}\cdot\nabla\widetilde{U}dS.
\]
By straightforward calculation, we have the following Rellich type
identity 
\begin{align}
 & \int_{B_{r}^{n+1}}\left[\nabla\cdot\left(|y|^{1-2s}\overrightarrow{\beta}(\widetilde{A}\nabla\widetilde{U}\cdot\nabla\widetilde{U})\right)-2\nabla\cdot\left(|y|^{1-2s}(\overrightarrow{\beta}\cdot\nabla\widetilde{U})\widetilde{A}\nabla\widetilde{U}\right)\right]dz\nonumber \\
= & \int_{B_{r}^{n+1}}\left[\nabla\cdot(|y|^{1-2s}\overrightarrow{\beta})(\widetilde{A}\nabla\widetilde{U}\cdot\nabla\widetilde{U})+\sum_{j,k,l=1}^{n+1}y^{1-2s}\beta_{l}\dfrac{\partial\widetilde{a}_{jk}}{\partial z_{l}}\dfrac{\partial\widetilde{U}}{\partial z_{j}}\dfrac{\partial\widetilde{U}}{\partial z_{k}}\right]dz\nonumber \\
 & -2\int_{B_{r}^{n+1}}\sum_{j,k,l=1}^{n+1}|y|^{1-2s}\widetilde{a}_{jk}\dfrac{\partial\beta_{l}}{\partial z_{k}}\dfrac{\partial\widetilde{U}}{\partial z_{j}}\dfrac{\partial\widetilde{U}}{\partial z_{k}}dz.\label{Equality for y}
\end{align}
Note that $\beta_{n+1}=\dfrac{y}{\mu(z)}$, so we have 
\begin{align}
 & \int_{B_{r}^{n+1}}\nabla\cdot(|y|^{1-2s}\overrightarrow{\beta})(\widetilde{A}\nabla\widetilde{U}\cdot\nabla\widetilde{U})dz\nonumber \\
= & \int_{B_{r}^{n+1}}(\nabla\cdot\overrightarrow{\beta})(\widetilde{A}\nabla\widetilde{U}\cdot\nabla\widetilde{U})dz+\int_{B_{r}^{n+1}}(1-2s)\dfrac{|y|^{1-2s}}{\mu(z)}\widetilde{A}\nabla\widetilde{U}\cdot\widetilde{U}dz.\label{eq:first term of the left hand side}
\end{align}

First, for the left hand sides in (\ref{Equality for y}), we use
the relations $\overrightarrow{\beta}\cdot\nu=r$, $\overrightarrow{\beta}\cdot\nabla\widetilde{U}=\dfrac{r(\widetilde{A}\nu)\cdot\nabla\widetilde{U}}{\mu(z)}$
on $\partial B_{r}^{n+1}$ and integrate them over $B_{r}^{n+1}$,
so we get 
\begin{align}
 & \int_{\partial B_{r}^{n+1}}|y|^{1-2s}(\widetilde{A}(x)\nabla\widetilde{U}\cdot\widetilde{U})(\overrightarrow{\beta}\cdot\nu)dS-2\int_{\partial B_{r}^{n+1}}\left(|y|^{1-2s}(\widetilde{A}\nu\cdot\nabla\widetilde{U}\right)(\beta\cdot\nabla\widetilde{U})dS\nonumber \\
= & r\int_{\partial B_{r}^{n+1}}|y|^{1-2s}(\widetilde{A}(x)\nabla\widetilde{U}\cdot\widetilde{U})dS-2r\int_{\partial B_{r}^{n+1}}|y|^{1-2s}\dfrac{\left|\widetilde{A}\nu\cdot\nabla\widetilde{U}\right|^{2}}{\mu(z)}dS\nonumber \\
= & rD'(r)-2r\int_{\partial B_{r}^{n+1}}|y|^{1-2s}\dfrac{\left|\widetilde{A}\nu\cdot\nabla\widetilde{U}\right|^{2}}{\mu(z)}dS.\label{eq:R1}
\end{align}

Second, we evaluate the right hand side of (\ref{Equality for y})
as follows. For the first term in the right hand side (RHS) of (\ref{Equality for y})
can be rewritten as \eqref{eq:first term of the left hand side} and
we estimate them separately. By using (\ref{Estimates for mu and beta}),
we have $\nabla\cdot\beta=n+1+O(r)$ for $z\in B_{1}^{n+1}$, which
implies 
\begin{equation}
\int_{B_{r}^{n+1}}(\nabla\cdot\beta)|y|^{1-2s}(\widetilde{A}(x)\nabla\widetilde{U}\cdot\nabla\widetilde{U})=(n+1+O(r))D(r),\label{eq:R2}
\end{equation}
and we know that $\beta_{n+1}=\dfrac{y}{\mu(z)}=y+\left(1-\dfrac{1}{\mu(z)}\right)y=y+O(|z|)y$,
with $|z|\leq r$, hence 
\begin{align}
 & \int_{B_{r}^{n+1}}\beta_{n+1}(1-2s)|y|^{-2s}(\widetilde{A}(x)\nabla\widetilde{U}\cdot\nabla\widetilde{U})dz\nonumber \\
= & \int_{B_{r}^{n+1}}(y+O(r)y)(1-2s)|y|^{-2s}(\widetilde{A}(x)\nabla\widetilde{U}\cdot\nabla\widetilde{U})dz\nonumber \\
= & (1-2s+O(r))D(r).\label{eq:R3}
\end{align}
For the second term in the RHS of (\ref{Equality for y}), we have
$\left|\beta_{l}\dfrac{\partial\widetilde{a}_{jk}}{\partial z_{l}}\right|\leq C|z|\leq Cr$
so that 
\begin{equation}
\sum_{j,k,l=1}^{n+1}\int_{B_{r}^{n+1}}|y|{}^{1-2s}\beta_{l}\dfrac{\partial\widetilde{a}_{jk}}{\partial z_{l}}\dfrac{\partial\widetilde{U}}{\partial z_{j}}\dfrac{\partial\widetilde{U}}{\partial z_{k}}=O(r)D(r),\label{eq:R4}
\end{equation}
For the last term in the RHS of (\ref{Equality for y}), Now, for
the last term in the RHS of (\ref{Equality for y}), from $\dfrac{\partial\beta_{l}}{\partial z_{k}}=\delta_{lk}+O(r)$
in a bounded region, it is easy to see that 
\begin{equation}
\int_{B_{r}^{n+1}}\sum_{j,k,l=1}^{n+1}|y|^{1-2s}\widetilde{a}_{jk}\dfrac{\partial\beta_{l}}{\partial z_{k}}\dfrac{\partial\widetilde{U}}{\partial z_{j}}\dfrac{\partial\widetilde{U}}{\partial z_{k}}dz=(1+O(r))D(r).\label{eq:R5}
\end{equation}
Finally, by plugging \eqref{eq:R1}, \eqref{eq:R2}, \eqref{eq:R3},
\eqref{eq:R4} and \eqref{eq:R5} into (\ref{Equality for y}), we
finish the proof of this lemma.
\end{proof}
Now, it is ready to prove the doubling inequality.
\begin{lem}
(Doubling inequality) Let $\widetilde{U}\in H^{1}(\mathbb{R}^{n+1},|y|^{1-2s})$
be a weak solution of $\nabla_{x,y}\cdot(|y|^{1-2s}\widetilde{A}(x)\nabla_{x,y}\widetilde{U})=0$
in $B_{1}^{n+1}$, then there exists a constant $C>0$ such that 
\begin{equation}
\int_{B_{2R}^{n+1}}|y|^{1-2s}\left|\widetilde{U}\right|^{2}dxdy\leq C\int_{B_{R}^{n+1}}|y|^{1-2s}\left|\widetilde{U}\right|^{2}dxdy,\label{Doubling inequality}
\end{equation}
whenever $B_{2R}^{n+1}\subset B_{1}^{n+1}$.\end{lem}
\begin{proof}
Since $H(r)$ and $D(r)$ are differentiable, so we can differentiate
$N(r)$ with respect to $r$, then we get 
\begin{equation}
N'(r)=N(r)\left\{ \dfrac{1}{r}+\dfrac{D'(r)}{D(r)}-\dfrac{H'(r)}{H(r)}\right\} .\label{eq:N'(r)}
\end{equation}
If we plug (\ref{eq:H(r)}), (\ref{eq:D(r)}), (\ref{eq:H'(r)}) and
(\ref{eq:D'(r)}) into (\ref{eq:N'(r)}) and use the Cauchy-Schwartz
inequality, then we can deduce that 
\begin{align*}
 & \dfrac{1}{r}+\dfrac{D'(r)}{D(r)}-\dfrac{H'(r)}{H(r)}\\
\geq & 2\left(\dfrac{\int_{\partial B_{r}^{n+1}}|y|^{1-2s}\dfrac{1}{\mu}\left|\widetilde{A}\nu\cdot\nabla\widetilde{U}\right|^{2}dS}{\int_{\partial B_{r}^{n+1}}|y|^{1-2s}\widetilde{U}\left(\widetilde{A}\nu\cdot\nabla\widetilde{U}\right)dS}-\dfrac{\int_{\partial B_{r}^{n+1}}|y|^{1-2s}\widetilde{U}\left(\widetilde{A}\nu\cdot\nabla\widetilde{U}\right)dS}{\int_{\partial B_{r}^{n+1}}|y|^{1-2s}\mu\left|\widetilde{U}\right|^{2}dS}\right)+O(1)\\
\geq & O(1)
\end{align*}
which implies 
\[
N'(r)\geq-CN(r)
\]
for some constant $C>0$. Moreover, for $R<1$, we integrate the above
inequality over $R$ to 1, then we have 
\[
\int_{R}^{1}\dfrac{d}{dr}\log N(r)dr\geq-C(1-R)\geq-C
\]
or 
\begin{equation}
N(R)\leq e^{-C}N(1).\label{eq:bound for N(r)}
\end{equation}
Note that (\ref{eq:H'(r)}) is equivalent to 
\[
\dfrac{d}{dr}\log\dfrac{H(r)}{r^{n+1-2s}}=2\dfrac{N(r)}{r}+O(1),
\]
where $O(1)$ is independent of $r$. After integrating over $(r,2r)$
and use \eqref{eq:bound for N(r)}, it is easy to see $H(2r)\leq CH(r)$
and integrate this quantity over $(0,R)$, which proves the doubling
inequality (\ref{Doubling inequality}). 
\end{proof}
\bibliographystyle{plain}
\bibliography{ref}

\end{document}